\documentclass[11pt,a4paper]{amsart}

\usepackage{amsfonts}
\usepackage{amsmath}
\usepackage{amssymb}
\usepackage{amsthm}
\usepackage{enumitem}
\usepackage{geometry}
\usepackage{latexsym}
\usepackage{hyperref}
\usepackage{mathrsfs}
\usepackage{tikz}
\usepackage{tikz-cd}
\usepackage{xy}
\usepackage{color}
\usepackage{verbatim}
\allowdisplaybreaks[4]

\hypersetup{colorlinks=true,citecolor=red,linkcolor=blue,urlcolor=black,}

\newtheorem{theorem}{Theorem}[section]
\newtheorem{lemma}[theorem]{Lemma}
\newtheorem{corollary}[theorem]{Corollary}
\newtheorem{proposition}[theorem]{Proposition}

\newtheorem{def-thm}[theorem]{Definition-Theorem}

\theoremstyle{definition}
\newtheorem{definition}[theorem]{Definition}

\newtheorem{example}[theorem]{Example}

\newtheorem{remark}[theorem]{Remark}

\DeclareMathOperator{\Supp}{Supp}

\DeclareMathOperator{\Exc}{Exc}

\DeclareMathOperator{\lct}{lct}

\DeclareMathOperator{\Sing}{Sing}

\DeclareMathOperator{\NE}{\overline{NE}} 
\DeclareMathOperator{\Nlc}{Nlc} 
\DeclareMathOperator{\rank}{rank}
\DeclareMathOperator{\loc}{loc} 

\newcommand{\cO}{\ensuremath{\mathcal{O}}}
\newcommand{\cF}{\ensuremath{\mathcal{F}}}
\newcommand{\cG}{\ensuremath{\mathcal{G}}}
\newcommand{\Q}{\ensuremath{\mathbb{Q}}}
\newcommand{\R}{\ensuremath{\mathbb{R}}}
\newcommand{\C}{\ensuremath{\mathbb{C}}}
\newcommand{\Z}{\ensuremath{\mathbb{Z}}}

\newcommand{\bN}{\ensuremath{\mathbf{N}}}
\newcommand{\bM}{\ensuremath{\mathbf{M}}}
\newcommand{\bir}{\ensuremath{\dashrightarrow}}
\newcommand{\bL}{\ensuremath{\mathbf{L}}}

\title{MMP for generalized foliated threefolds of rank one}
\author{Mengchu Li}

\address{School of Mathematical Sciences, Fudan University, Shanghai 200433, China}
\email{mengchuli21@m.fudan.edu.cn}

\date{\today}

\begin{document}

\begin{abstract}
    We establish the minimal model program (MMP) for generalized foliated threefolds $(X, \mathcal{F}, B, \mathbf{M})$ of rank 1, extending the result of Cascini and Spicer in \cite{CS20}. As an application of the generalized foliated MMP, we prove a base-point-free theorem for foliated triples on threefolds. We also prove the ACC for log canonical thresholds for generalized foliated threefolds of rank 1.
\end{abstract}

\maketitle

\setcounter{tocdepth}{1}
\tableofcontents

\section{Introduction}

The minimal model program (MMP) plays a central role in birational geometry. It predicts that any varieties with mild singularities is either uniruled or birational to a minimal model -- namely, a birational model with nef canonical divisor. In recent years, many results shows that an analogue picture holds for the birational geometry of foliations. One may replace the canonical divisor $K_X$ for a variety $X$ with the canonical divisor $K_\cF$ for a foliation $\cF$ on $X$, and study whether the results of the classical MMP extend to the foliated case.

In lower-dimensional cases, many results concerning the foliated minimal model program, such as the cone theorem, contraction theorem, and the existence of flips, are already known. The MMP for foliated surfaces was developed by McQuillan and Brunella~\cite{McQ08,Bru15}. For corank 1 foliations on threefolds, the program was established by Cascini, Spicer, etc.~\cite{CS21,Spi20,SS21,CM24}. For rank 1 foliations on threefolds, the program is initiated in the works of McQuillan and Bogomolov~\cite{McQ04, BM16}, and established by Cascini and Spicer~\cite{CS20}. For the higher-dimensional case, although the general theory of foliated MMP remains largely open, the program has been established for algebraically integrable foliations, i.e. foliations induced by rational dominant maps~\cite{ACSS22,cascini_mmp_2023,CHLX23,liu_minimal_2024}.

Just as Birkar and Zhang extended the notion of pairs in classical birational geometry to generalized pairs in \cite{BZ16}, one can also generalize foliated triples $(X,\cF,B)$ to \textit{generalized foliated quadruples} $(X,\cF,B,\bM)$, which is first introduced in~\cite{LLM23}. See Definition~\ref{def:GFQS} for the definition of generalized foliated quadruples. 
Recent works have shown that the notion of generalized foliated quadruples is technically important in the context of foliated minimal model program. For instance, the notion of generalized foliated quadruples naturally appears in the study of the canonical bundle formula for foliations (see e.g.~\cite[Theorem 1.3]{LLM23}).
The minimal model program for algebraically integrable generalized foliated quadruples is established in~\cite{CHLX23,liu_minimal_2024}, where numerous results concerning foliations and generalized pairs are obtained using the language and technique of generalized foliated quadruples.
Furthermore, the minimal model program for algebraically integrable adjoint foliated structure, established in \cite{cascini_minimal_2024,cascini_finite_2025}, also heavily relies on the notion of generalized foliated quadruples. Thus, it is natural to ask whether the minimal model program runs for generalized foliated quadruples without algebraically integrability condition, at least for the lower-dimensional case, and what kind of extensions of existing results might arise from this approach. 

\subsection{Main results}
In this paper, we focus on foliations of rank 1 on threefolds. Our main goal is to develop the minimal model program for generalized foliated threefolds of rank one. The following is our first main result, which is a generalization of \cite[Theorem 1.2]{CS20}.

\begin{theorem}\label{MainThm:FoliatedMMP}
    Let $X$ be a normal projective threefold, and let $(X,\cF,B,\bM)$ be an lc generalized foliated quadruple of rank $1$ on $X$. Assume that $X$ is $\Q$-factorial klt. Then
    \begin{enumerate}
        \item[(1)] If $K_\cF + B + \bM_X$ is pseudo-effective, then we may run a $(K_\cF + B +\bM_X)$-MMP which terminates on a minimal model $\phi: X \bir X'$ of $(X,\cF,B,\bM)$.
        \item[(2)] If $K_\cF + B + \bM_X$ is not pseudo-effective, then we may run a $(K_\cF + B +\bM_X)$-MMP which terminates on a Mori fiber space.
    \end{enumerate}
\end{theorem}

One of the key steps in proving the above theorem on the generalized foliated MMP of rank one on threefolds is to establish the cone theorem for generalized foliated threefolds as follows.

\begin{theorem}\label{MainThm:conethmgfq}
    Let $X$ be a normal projective threefold and $(X,\cF,B,\bM)$ be a  generalized foliated quadruple of rank 1 on $X$. Then there exist rational curves $\{C_j\}_{j\in \Lambda} $ which are tangent to $\cF$ satisfying
    \[
    0 < -(K_\cF + B + \bM_X) \cdot C_j \leq 2,
    \]
    for each $j \in \Lambda$. Moreover,
    \[
    \NE (X) = \NE(X)_{K_\cF + B + \bM_X\geq 0} + \NE (X)_{\Nlc (X,\cF,B,\bM)} + \sum_{j \in \Lambda} \R_{\geq 0}[C_j],
    \]
    where $\NE (X)_{\Nlc (X,\cF,B,\bM)}$ is the subcone of $\NE(X)$ spanned by the images of $\NE(W) \to \NE(X)$ where $W$ is a non-lc centre of $(X,\cF,B,\bM)$.
\end{theorem}

Note that we can show the length of extremal rays can be bounded by $2$, rather than the commonly expected bound $2\dim X = 6$. Also in the cone theorem stated above, we do not assume that either $K_\cF$ or $B$ is $\R$-Cartier. We also note that Cascini and Spicer proved a version of the cone theorem for rank one foliated triples $(X, \cF, B)$ in arbitrary dimension, under the assumption that both $K_\cF$ and $B$ are $\R$-Cartier (see \cite[Theorem 4.8]{CS24}). In contrast, our result applies only to threefolds. See Remark~\ref{rmk:WhyCS24notwork} for details.

As an application of the MMP above, we prove a version of  base-point-free theorem for foliated triples on threefolds, which generalize the result of \cite[Theorem 1.1]{CS25bpf}.

\begin{theorem}\label{MainThm:BPF}
    Let $X$ be a normal projective threefold, and let $(X,\cF,B)$ be an lc foliated triple of rank $1$ on $X$. Assume that $X$ is $\Q$-factorial klt. Let $D$ be a nef $\R$-divisor on $X$ such that $D - (K_\cF + B)$ is an ample $\R$-divisor. Then $D$ is semi-ample.
\end{theorem}

In the above statement, let $A = D - (K_\cF + B)$. In \cite[Theorem 1.1]{CS25bpf}, the same result is proved under the assumptions that $(X, \cF, B + A)$ is lc and that both $B$ and $A$ are $\Q$-divisors. Our proof instead relies on the generalized foliated structure; see Theorem~\ref{Thm:BPFgfq} for the generalized version.

In the end of this paper, we also give a proof of the ascending chain condition (ACC) for log canonical thresholds for generalized foliated threefolds of rank 1 as a corollary of Theorem~\ref{MainThm:FoliatedMMP}, generalizing the result of Y. Chen in \cite{Che22}. See Definition~\ref{def:GFLCT} for the definition of generalized foliated log canonical thresholds.

\begin{theorem}\label{MainThm:ACCforLCT}
    Let $I\subset [0,+\infty)$ be a DCC set. Then there exists an ACC set $J$ depending only on $I$ such that if $X$ is a normal projective threefold, $(X,\cF,B,\bM)$ is an lc generalized foliated quadruple of rank $1$ on $X$, and there exist an $\R$-divisor $D $ and a b-divisor $\bN$ on $X$ satisfying
    \begin{enumerate}
        \item[(1)] $D + \bN_X$ is $\R$-Cartier,
        \item[(2)] the coefficients of $B$ and $D$ belong to $I$, and 
        \item[(3)] both $\bM$ and $\bN$ are $I$-linear combinations of b-nef b-Cartier b-divisors on $X$,
    \end{enumerate}
    then $\lct (X,\cF,B,\bM;D,\bN) \in J$.
\end{theorem}

Various results have been proved concerning the ACC for log canonical thresholds for foliations. Besides the classical case ($\cF = T_X$), proved in \cite{HMX14} for pairs or \cite{BZ16} for generalized pairs, \cite{Che22} proved the ACC for log canonical thresholds of foliated triples in dimension $\leq 3$. For algebraically integrable foliations, \cite{DLM23} proved ACC for log canonical thresholds for foliated triples, which is generalized in \cite{CHLX23} to the case of generalized foliated quadruples.

\subsection{Idea of the proof} First we state the main ideas and the sketch of the proofs of Theorem~\ref{MainThm:FoliatedMMP} and Theorem~\ref{MainThm:conethmgfq}. Assume $X$ is $\Q$-factorial. A key observation is that if $C$ is an $\cF$-invariant curve such that $(X,\cF,B,\bM)$ is lc at the generic point of $C$, then the degree of $\bM_X$ along $C$ is non-negative  (see Proposition \ref{proposition:bdivoncurve}).
As a corollary, each $(K_\cF + B + \bM_X)$-negative extremal ray considered in the cone theorem is in fact $(K_\cF + B)$-negative.
Thus to prove the cone theorem, we may reduce to the foliated triple case and follow the similar approach as in \cite{CS24}. Using the bend and break lemma and adjunction on surfaces, we can prove the cone theorem assuming the ambient variety is $\Q$-factorial (see Theorem~\ref{theorem:conethmgfq,Q-fac}). In particular, we need a version of the cone theorem for foliated surfaces. To this end, we consider intersection theory on normal surfaces in the sense of Mumford, and state a version of the cone theorem for surface num-gfqs (Definition~\ref{def:NumGFQ}). Section~\ref{section:normalsurfaces} is devoted to this purpose. 
Moreover, any $(K_\cF + B + \bM_X)$-MMP is a step of $(K_\cF + B)$-MMP again by Proposition~\ref{proposition:bdivoncurve}, and the existence of contractions and flips can be reduced to the results in \cite{CS20}, which proves Theorem~\ref{MainThm:FoliatedMMP}. Theorem~\ref{MainThm:conethmgfq} then follows directly from the existence of the MMP.

In the setting of Theorem~\ref{MainThm:BPF}, let $A = D - (K_\cF + B)$. Since Bertini-type theorems fail for foliations (see e.g. \cite[Example 3.4]{DLM23}), the foliated triple $(X, \cF, B + A)$ is not log canonical in general. However, we may regard $A$ as a b-divisor $\overline{A}$, under which the generalized foliated quadruple $(X, \cF, B, \overline{A})$ is lc. We then generalize the approach of~\cite{CS25bpf}
to conclude the result. The proof of Theorem~\ref{MainThm:ACCforLCT} follows similar ideas as in \cite{Che22} or \cite[Proposition 4.6]{LMX24}. In our case, we need a more detailed study on coefficients appearing in adjunction formulas when working with b-divisors, as developed in Section~\ref{section:ACCforLCT}.

\subsection*{Acknowledgement} The author would like to express his sincere gratitude to his advisor, Professor Meng Chen, for his support and encouragement. Part of this work was completed during his time as a visiting student at Imperial College London. The author would like to thank Professor Paolo Cascini for his hospitality and numerous helpful discussions. He also thanks Jihao Liu for his help and for suggesting this problem. The author is grateful to Hexu Liu, Yuting Liu, Roktim Mascharak, and Minzhe Zhu for their valuable discussions. 
The author also appreciates the referee's valuable suggestions and corrections.
The author is supported by the China Scholarship Council (No. 202306100155) and the Fudan Elite PhD Program.

\section{Preliminaries}

We work over the field of complex numbers $\mathbb{C}$. All varieties are assumed to be quasi-projective over $\mathbb{C}$. We adopt the standard notations and definitions in birational geometry as in \cite{KM98} and \cite{BCHM10}, and use them freely throughout the paper. The coefficient of a prime divisor $D$ in an $\R$-divisor $M$ is denoted by $\mu_D M$.

\subsection{Basic definitions} We begin by recalling some basic definitions and notations related to foliations and b-divisors, and define generalized foliated quadruples, following \cite{CS21,HL23,CHLX23}.

\begin{definition}[Foliations]

Let $X$ be a normal variety. A \textit{foliation} $\cF$ on $X$ is a coherent subsheaf $\cF\subset T_X$ such that
\begin{enumerate}
    \item[(1)] $\cF$ is saturated in $T_X$, i.e. $T_X/\cF$ is torsion free, and 
    \item[(2)] $\cF$ is closed under the Lie bracket.
\end{enumerate}
The \textit{rank} of $\cF$ is the rank of $\cF$ as a coherent sheaf and is denoted by $\rank \cF$. The \textit{canonical divisor} of $\cF$ is a Weil divisor $K_\cF$ such that $\cO_X(-K_\cF)\cong \det(\cF)$. 

Let $X$ be a normal variety and $\cF$ be a foliation of rank $r$ on $X$. We may associate $\cF$ to a morphism:
\[
\phi: \Omega_X^{[r]} \to \cO_X(K_\cF),
\]
defined by taking the double dual of the $r$-wedge product of the map $\Omega_X^1 \to \cF^*$, which is induced by the inclusion $\cF \subset T_X$. This yields a map 
\[
\phi': (\Omega_X^{[r]}\otimes \cO_X(-K_\cF))^{**} \to \cO_X, 
\]
and we define the \textit{singular locus}, denoted by $\Sing \cF$, to be the cosupport of the image of $\phi'$. 

\end{definition}

\begin{definition}[Pullbacks]
    Let $X$ be a normal variety and let $\cF$ be a foliation on $X$. Let $f: Y \bir X$ be a dominant map. We denote $f^{-1}\cF$ to be the \textit{pullback} of $\cF$ as in \cite[3.2]{Dru21}. We also say $f^{-1}\cF$ is the \textit{induced foliation} of $\cF$ on $Y$. If $\cF = 0$, then we say $f^{-1}\cF$ is the foliation induced by $f$.  If a foliation $\cG$ is induced by a dominant map, then we say $\cG$ is \textit{algebraically integrable}. If $g: X \bir X'$ is a birational map, then we define the \textit{pushforward} $g_* \cF$ of $\cF$ on $X'$ to be $(g^{-1})^{-1}(\cF)$.
\end{definition}

\begin{definition}[Invariance]

Let $S$ be a subvariety of $X$. We say $S$ is \textit{$\cF$-invariant} if for any open subset $U\subset X$ and any section $\partial \in H^0(U,\cF)$, we have 
$ \partial(\mathcal{I}_{S \cap U}) \subset \mathcal{I}_{S \cap U}$, where $\mathcal{I}_{S \cap U}$ is the ideal sheaf of $S \cap U$. If $D\subset X$ is a prime divisor, we define $\epsilon(D) = 0$ if $D$ is $\cF$-invariant and $\epsilon(D) = 1$ if $D$ is not $\cF$-invariant. 

\end{definition}

\begin{definition}[b-divisors]

Let $X$ be a normal variety, and let $X \bir X'$ is a birational map.
For any valuation $\nu$ over $X$, we denote by $\nu_{X'}$ the center of $\nu$ on $X'$. A \emph{b-divisor} $\bM$ is a formal sum $\bM = \sum r_\nu \nu$ where $\nu$ are valuations over $X$ and $r_\nu \in \R$, such that $\nu_X$ is a divisor only for finitely many $\nu$. The trace of $\bM$ on $X'$ is defined as 
\[
\bM_{X'} := \sum_{\nu_{X'} \text{ is a divisor}} r_\nu \nu_{X'},
\]
which is an $\R$-divisor on $X'$. If $\bM_{X'}$ is $\R$-Cartier and $\bM_{Y'} = g^* \bM_{X'}$ for any birational morphism $g: Y' \to X'$, then we say $\bM$ \emph{descends} to $X'$, or say $\bM$ is the \emph{closure} of $\bM_{X'}$, and we write $\bM = \overline{\bM_{X'}}$.

A b-divisor $\bM$ on $X$ is called \textit{b-$\R$-Cartier} if there exists a birational morphism $f:Y \to X$ such that $\bM = \overline{\bM_Y}$. 
In this case, if $\bM_Y$ is nef on $Y$, then we say that $\bM$ is a \textit{b-nef} b-\R-Cartier divisor. If $\bM_Y$ is a Cartier divisor (resp. a $\Q$-Cartier $\Q$-divisor), then we say that $\bM$ is \emph{b-Cartier} (resp. \emph{b-$\Q$-Cartier}).

\end{definition}

\begin{definition}[Restricted b-divisors]\label{def:Restrictedbdiv}
    
Let $\bM$ be a b-\R-Cartier divisor on $X$. Let $S$ be a prime divisor on $X$ and $\nu: S^\nu \to S$ be the normalization of $S$. Let $f: Y \to X$ be a log resolution of $(X,S)$ such that $\bM$ descends to $Y$. Let $S_Y = f^{-1}_*S$, then there exists a birational morphism $g: S_Y \to S^\nu$ such that $f|_{S_Y} = \nu \circ g$. We define the \textit{restricted b-divisor} $\bM|_{S^\nu}$ of $\bM$ to $S^\nu$ to be the closure of $ \bM_Y |_{S_Y} $. Here $\bM_Y|_{S_Y}$ is determined up to $\R$-linear equivalence. Note that if $\bM$ is nef b-\R-Cartier, then so is $\bM|_{S^\nu}$. Moreover, if $\bM$ is b-Cartier, then so is $\bM|_{S^\nu}$.

\end{definition}

\begin{definition}
Fix $I\subset \R$. We say a b-$\R$-Cartier b-divisor $\bM$ on $X$ is an \textit{$I$-linear combination} of b-Cartier b-divisors, if $\bM = \sum \mu_i\bM^i$ for some b-Cartier b-divisors $\bM^i$ on $X$ and $\mu_i \in I$. Sometimes for simplicity, we may just say that the coefficients $\mu_i$ of $\bM^i$ belong to $I$, or just $\bM \in I$. We also say $\bM_i$ \textit{appears in} $\bM$ if $\mu_i \neq 0$. Note that if $S$ is a prime divisor on $X$ and  $\bM = \sum \mu_i \bM^i$, then $\bM|_{S^\nu} = \sum \mu_i \bM^i|_{S^\nu}$. Thus $\bM\in I$ implies $\bM|_{S^\nu}\in I$.
\end{definition}

\begin{definition}[Generalized foliated quadruples]\label{def:GFQS}
    
A \textit{generalized foliated sub-quadruple} $(X,\cF,B,\bM)$ (\textit{sub-gfq} for short) consists of a normal variety $X$, a foliation $\cF$ on $X$, an \R-divisor $B$ on $X$, a b-nef b-\R-Cartier b-divisor $\bM$ on $X$, such that $K_\cF + B + \bM_X$ is $\R$-Cartier. If $B\geq 0$, we say $(X,\cF,B,\bM)$ is a \textit{generalized foliated quadruple} (\textit{gfq} for short). If $\bM = \overline{0}$ (and $B \geq 0$), we drop $\bM$ and say $(X,\cF,B)$ is a \textit{foliated sub-triple} (\textit{foliated triple}). 

\end{definition}

\begin{definition}[Singularities]
Let $(X, \cF,B,\bM)$ be a sub-gfq. For any prime divisor $E$ over $X$, suppose $f: Y\to X$ is a birational morphism such that $E$ is a prime divisor on $Y$, then we may write 
\[
K_{\cF_Y} + B_Y + \bM_Y = f^* (K_\cF + B + \bM_X),
\]
where $\cF_Y = f^{-1}\cF$. We define $a(E,\cF,B,\bM) := -\mu_E B_Y$ to be the \textit{discrepancy} of $E$ with respect to $(X,\cF,B,\bM)$. It is easy to see that the definition of $a(E,\cF,B,\bM)$ is independent of the choice of $Y$. If $\bM = \overline{0}$ (resp. $B = 0$), we will drop $\bM$ (resp. $B$) in the notation of discrepancies. If $\cF = T_X$, the definition of the discrepancies coincides with the classical definition of discrepancies for generalized pairs as in \cite{BZ16}. 

For a (sub)-gfq $(X, \cF, B, \bM)$, we say $(X,\cF,B,\bM)$ is \textit{(sub)-log canonical} (\textit{(sub)-lc} for short) if $a(E,\cF,B,\bM) \geq -\epsilon(E)$ for any prime divisor $E$ over $X$. We say $(X,\cF,B,\bM)$ is \textit{(sub)-canonical} (resp. \textit{(sub)-terminal}) if $a(E,\cF,B,\bM) \geq 0$ (resp. $>0$) for any prime divisor $E$ exceptional over $X$. An \textit{lc place} (resp. \textit{non-lc place}) of $(X,\cF,B,\bM) $ is a prime divisor $E$ over $X$ with $a(E,\cF, B,\bM) = -\epsilon(E)$ (resp. $a(E,\cF, B,\bM) < -\epsilon(E)$) . An \textit{lc center} (resp. \textit{non-lc center}) of $ (X,\cF,B,\bM) $ is the center of an lc place (resp. non-lc place) of $(X,\cF,B,\bM)$ on $X$. 

Let $(X, \cF,B,\bM)$ be a gfq, and $P\in X$ be a not necessarily closed point of $X$. We say that $(X,\cF,B ,\bM)$ is lc (resp. canonical, terminal) at $P$ if for every divisor $E$ over $X$ whose center is the Zariski closure $\overline{P}$ of $P$, the discrepancy of $E$ is $\geq -\epsilon(E)$ (resp. $\geq 0$, $>0$). 

\end{definition}

\begin{definition}[Generalized foliated log canonical thresholds]\label{def:GFLCT}
    Let $(X,\cF,B,\bM)$ be a gfq. Let $D\geq 0$ be an effective divisor on $X$ and $\bN$ be a b-nef b-divisor on $X$ such that $D + \bN_X$ is $\R$-Cartier. We define the \textit{generalized foliated log canonical threshold} of $(D, \bN)$   with respect to $(X, \cF, B ,\bM)$ to be 
    \[
    \lct (X,\cF,B,\bM;D,\bN) := \sup \{ \,s \in \R \mid (X,\cF,B + sD, \bM + s\bN)\, \text{is lc}\, \}.
    \]
    If $\cF = T_X$, we may omit $\cF$, and the definition coincides with the generalized lc thresholds for generalized pairs (\cite[Definition 4.3]{BZ16}).
\end{definition}

\subsection{Finite morphism}

In this subsection we consider the finite pullback of gfqs. First we define the pullbcak of b-$\R$-Cartier b-divisor under finite morphisms.

\begin{definition}[Finite pullback]
    Let $\sigma:X' \to X$ be a finite morphism between normal varieties and let $\bM$ be a b-$\R$-Cartier b-divisor on $X$. If $Y \to X$ is a birational morphism such that $\bM$ descends to $Y$, then let $Y' $ be the normalization of the main component of $X'\times_X Y$, and let $\sigma': Y' \to Y$ be the induced morphism, and we define the pullback $\sigma^*\bM$ of $\bM$ under $\sigma$ to be the closure of the $\R$-Cartier divisor $\sigma'^*\bM_{Y}$ on $Y'$. It is easy to see that the definition is independent of the choice of $Y$.
\end{definition}

\begin{lemma}\label{lemma:finitepullbackMX}
    Let $\sigma:X' \to X$ be a finite morphism between normal varieties and let $\bM$ be a b-$\R$-Cartier b-divisor on $X$. Let $\bM' = \sigma^*\bM$ be the pullback b-divisor of $\bM$ along $\sigma$. If $\bM_X$ is $\R$-Cartier, then $\bM'_{X'} = \sigma^* \bM_X$.
\end{lemma}

\begin{proof}
    Let $f: Y \to X$ be a birational morphism such that $Y $ is smooth and $\bM$ descends to $Y$. Let $Y'$ be the normalization of the main component of $X'\times_X Y$, and let $\sigma': Y' \to Y$, $g: Y' \to X'$ be the induced morphisms. We may write $\bM_Y + E = f^* \bM_X$, where $E$ is exceptional over $X$. By the definition of the pullback, $\bM'_{Y'} = \sigma'^*\bM_Y$. Then we have $\bM'_{Y'} + \sigma'^*E= \sigma'^*f^*\bM_X = g^*\sigma^* \bM_X$, and hence $\bM_{X'} + g_* \sigma'^* E = \sigma^*\bM_X$. Note that since $\sigma$, $\sigma'$ are both finite, for each prime divisor $D$ exceptional over $X$, $\sigma'^{-1}(D)$ is exceptional over $X'$. Thus $g_*\sigma'^*E = 0$ and we may conclude.
\end{proof}

\begin{lemma}\label{lemma:GFQfinite}
    Let $(X,\cF,B,\bM)$ be a sub-gfq and $\sigma: X' \to X$ be a finite morphism between normal varieties. Let $Z'$ be an irreducible subvariety of $X'$ and let $Z = \sigma(Z')$. Let $\cF' = \sigma^{-1}\cF$ and $\bM' = \sigma^*\bM$. We may write 
    \[
    K_{\cF'} + B' + \bM'_{X'} := \sigma^*(K_\cF + B + \bM_X),
    \]
    which defines a sub-gfq $(X',\cF',B',\bM')$. Then
    \begin{enumerate}
        \item[(1)] $(X,\cF,B,\bM)$ is sub-lc at the generic point of $Z$ if and only if $(X',\cF',B',\bM')$ is sub-lc at the generic point of $Z'$. Moreover, $Z$ is an lc center of $(X,\cF,B,\bM)$ if and only if $Z'$ is an lc center of $(X',\cF',B',\bM')$.
        \item[(2)] If $(X,\cF,B,\bM)$ is sub-canonical (resp. sub-terminal) at the generic point of $Z$, then  $(X',\cF',B',\bM')$ is sub-canonical (resp. sub-terminal) at the generic point of $Z'$.
    \end{enumerate}
\end{lemma}

\begin{proof}
    Let $f: Y \to X$ be a birational morphism such that $\bM$ descends to $Y$ and $E$ be a prime divisor on $Y$ whose center on $X$ is $Z$. Let $Y'$ be the normalization of the main component of $X' \times_X Y$ and $\sigma': Y' \to Y$ and $g: Y' \to X'$ be induced morphisms. Let $E'$ be a prime divisor on $Y'$ which dominates $E$ and has center $Z'$ on $X'$.

    Let $g: Y' \to X'$ be a birational morphism and $E'$ be a $g$-exceptional divisor on $Y'$ centering on $Z'$. By \cite[Lemma 2.22]{Kol13}, there exist a birational morphism $f: Y \to X$ and a finite morphism $\sigma': Y' \to Y$ such that $E = \sigma'(E')$ is a $f$-exceptional prime divisor on $Y$ and is centered on $Z$. After replacing $Y$, $Y'$ by higher models, we may assume that $\bM$ (resp. $\sigma^*\bM$) descends to $Y$ (resp. $Y'$).
    
    In both cases above, let $ \cF_Y = f^{-1}\cF $ and $\cF_{Y'} =g^{-1}\cF' $. By \cite[Lemma 3.4]{Dru21} (see also \cite[Proposition 2.2]{CS20}),
    \[
    K_{\cF_{Y'}} = \sigma^*K_{\cF_Y}  + \sum \epsilon(\sigma(D))(r_D-1)D,
    \]
    where the sum runs over all prime divisors on $Y'$ and $r_D$ is the ramification index of $\sigma'$ along $D$. Since $\epsilon(E) = \epsilon(E')$ and $\bM'_{Y'} = \sigma'^*\bM_Y$, following the same computation of \cite[Proposition 5.20]{KM98} (see also \cite[Proposition 3.15]{HJLL24}), we have
    \[
    a(E',\cF',B',\bM') + \epsilon(E') = r_{E'} (a(E,\cF,B,\bM) + \epsilon(E)).
    \]
    Thus $a(E',\cF',B',\bM') \geq -\epsilon(E')$ (resp. $= -\epsilon (E')$) if and only if $a(E,\cF,B,\bM) \geq -\epsilon(E)$ (resp. $= -\epsilon (E)$). If $a(E,\cF,B,\bM) \geq 0$ (resp. $>0$), then $a(E',\cF',B',\bM') \geq 0$ (resp. $>0$). Note that $E$ is $f$-exceptional if and only if $E'$ is $g$-exceptional. Therefore (1) and (2) follow.
\end{proof}

\subsection{Divisorial adjunction}

We give the subadjunction formula for gfqs along prime divisors.

\begin{proposition}\label{prop:subadj,primediv}
    Let $X$ be a normal variety, $S$ be an prime divisor on $X$. Let $\nu: S^\nu \to S$ be the normalization of $S$. Let $(X,\cF,B,\bM)$ be a gfq of rank $r$ such that $S$ is not contained in $\Supp B$. Assume that $K_\cF + B$  and $ \epsilon(S)S $ are $\R$-Cartier. Then there exists a restricted foliation $\cF_S$ on $S^\nu$ with  $\rank \cF_S = r - \epsilon(S)$, an \R-divisor $B_S\geq 0$ on $S^\nu$, such that
    \[
    (K_\cF + B + \epsilon(S) S + \bM_X)|_{S^\nu} = K_{\cF_{S}} + B_S + \bM^S_{S^\nu},
    \]
    where $\bM^S = \bM|_{S^\nu}$ is the restricted b-divisor.
\end{proposition}

\begin{proof}

    By \cite[Proposition-Definition 3.7]{CS24}, there exists a restricted foliation $\cF_S$ on $S^\nu$ of rank $r -\epsilon(S)$ and  an \R-divisor $B'_{S}\geq 0$ on $S^\nu$ such that 
    \[
    (K_\cF + B + \epsilon(S)S)|_{S^\nu} = K_{\cF_S} + B'_S.
    \]
    Note that \cite{CS24} only dealt with the case when $B$ is a $\Q$-divisor, but we may write $B$ as a $\R$-linear combination of $\Q$-divisors $B_i \geq 0$ such that $K_\cF + B_i$ is $\Q$-Cartier, so the same argument holds when $B$ is an $\R$-divisor. Since $\bM_X$ is $\R$-Cartier, it suffices to prove that $\bM_X|_{S^\nu} \geq \bM_{S^\nu}^S$. Let $f:Y\to X$ be a log resolution of $(X,S)$ such that $\bM$ descends to $Y$. By the negativity lemma, we may write $\bM_Y + E = f^*\bM_X$ for an $f$-exceptional $\R$-divisor $E \geq 0$. Let $S_Y $ be the strict transform of $S$ to $Y$ and $g: S_Y \to S^\nu$ be the induced morphism. Therefore,
    \begin{align*}
        \bM_{S^\nu}^S &= g_*\bM_{S_Y}^S =g_*(\bM_Y|_{S_Y}) = g_*((f^*\bM_X -E)|_{S_Y}) \\&= \bM_X|_{S^\nu}- g_*(E|_{S_Y}) \leq \bM_X|_{S^\nu}. \qedhere
    \end{align*}
\end{proof}

\subsection{Sets} Let $I\subset \R$ be a set. We say $I$ satisfies DCC (descending chain condition) if it contains no strictly decreasing infinite sequence of numbers. We say $I$ satisfies ACC (ascending chain condition) if it contains no strictly increasing infinite sequence of numbers. In both cases, we also just say that $I$ is a DCC or an ACC set.

\begin{definition}

Let $I \in [0,+\infty)$ be a set. Define 
\[
I_+ := \{0\}\cup \left\{\, \sum_{k=1}^n i_k \,\Bigg|\, i_k \in I, n \in \Z_{>0} \, \right\},
\]
It is clear that $I \subset I_+$. Note that if $I\subset [0,+\infty)$ is a DCC set, then $I_+$ is also a DCC set.

\end{definition}

\begin{lemma}\label{lemma:finiteset}
Let $I\in [0,+\infty)$ be a DCC set and let $\Omega \subset (0,+\infty)$ be a finite set. Then the set 
\[
J :=\left\{\, i\in I \,\Bigg|\, mi + \sum_{k=0}^n m_ki_k \in \Omega,\, \textrm{for some } i_k\in I,\, m\in \Z_{>0},\, n,m_k \in \Z_{\geq 0}  \,\right\}
\]
is finite.
\end{lemma}
\begin{proof}
    We may assume that $\Omega = \{c\}$ for some fixed $c>0$. For $i\in J$, we may write $mi + \sum_{k=0}^n m_ki_k = c$ for some $i_k\in I$, $ m\in \Z_{>0}$, $n,m_k \in \Z_{\geq 0}$. Then 
    \[
    i = \frac{1}{m} \left(c -\sum_{k=0}^n m_ki_k\right).
    \]
    The LHS belongs to $I$, which is a DCC set. Note that $\sum_{k=0}^n m_ki_k \in I_+$, The RHS belongs to an ACC set only depends on $I$ and $c$. Thus $i$ belongs to a finite set.
\end{proof}

\subsection{Convex geometry} We recall some terminology and theorems from convex geometry that will be used in the following proof of the cone theorems. For more details, see \cite[$\S$18]{Roc70}.

\begin{definition}[Exposed rays]
    
A convex cone $K \subset \R^n$ is said to be \emph{strongly convex} if $K \cap (-K) = \{0\}$. Let $K\subset \R^n$ be a strongly convex cone containing more than just the origin. A ray $R$ is said to be \textit{exposed} if there exists a hyperplane $H$ that meets $K$ exactly along $R$. The hyperplane $H$ is called a \textit{supporting hyperplane} of $R$. A \textit{supporting function} of $R$ is a linear functional $f\in (\R^n)^*$ such that $f(K) \geq 0$ and for $\xi\in K$, $f(\xi) = 0$ if and only if $\xi\in R$. 

\end{definition}

In particular, any exposed ray is extremal, but not vice versa. Assume that $K$ is a closed strongly convex cone containing more than just origin. Then $K$ is generated by its extremal rays, and it is the closure of the subcone generated by its exposed rays. In particular, every extremal ray of $K$ is the limit of a sequence of exposed rays of $K$ (see \cite[Theorem 18.6]{Roc70}).

\begin{lemma}\label{lemma:ExtRayPullback}
    Let $h: V \to W$ be a linear map between finite dimensional $\R$-vector spaces, and let $C \subset V$ and $D \subset W$ be two closed strongly convex cones containing more than just the origin. Assume that $h(C) \subset D$. Let $R$ be an extremal ray of $D$ contained in $h(C)$. Then there exists an extremal ray $R'$ of $C$ such that $R' = h(R)$. Moreover, let $L\in W^*$, then $R$ is $L$-trivial (resp. $L$-negative) if and only if $R'$ is $h^*L$-trivial (resp. $h^*L$-negative), where $h^* : W^* \to V^*$ is the dual map.
\end{lemma}

\begin{proof}
    Since $R$ is contained in $h(C)$, we may assume that $R = h(R'')$ for a ray $R''$ contained in $C$. Since $C$ is closed and strongly convex, it is generated by extremal rays. We may assume $R''$ is generated by $\sum a_i \xi_i$, where $a_i >0$ and $\xi_i$ generates an extremal ray of $C$ for each $i$. Then $R$ is generated by $\sum a_i h(\xi_i)$. Since $R$ is extremal, we have $R$ is generated by $h(\xi_j)$ for some $j$. Let $\xi' = \xi_j$, and we may take $R'$ to be the extremal ray of $C$ generated by $\xi'$. Note that $\langle h^*L,\xi' \rangle = \langle L,h(\xi') \rangle$ by the definition of $h^*$, the last assertion follows.
\end{proof}

In the remainder of the paper, we use Lemma~\ref{lemma:ExtRayPullback} frequently for $h: N_1(Y/S)_\R \to N_1(X/S)_\R$ and the cones of curves, induced by proper morphisms $f: X \to S$ and $g:Y\to X$. By the projection formula, the dual of $h$ is just the pullback of the divisor class $g^*: N^1(X/S)_\R \to N^1(Y/S)_\R$.

\begin{definition}[Locus]
    Let $X$ be a normal proper variety, and let $R$ be a ray in the cone of curves $\NE(X)$. We define the \emph{locus} of $R$, denoted by $\loc R$, to be the union of all curves $C$ on $X$ such that $[C]\in R$.
\end{definition}

\subsection{Bend and break} 

We recall the following foliated bend and break lemma originated from \cite{Miy87}, and stated in \cite{Spi20} as follows. 

\begin{theorem}[{\cite[Corollary 2.28]{Spi20}}]\label{thm:B&B}
    Let $X $ be a normal projective variety of dimension n,  $(X,\cF,B)$ be a foliated triple, and $N, D_1, \dots ,D_n$ be nef $\R$-divisors on $X$. Assume that 
    \begin{enumerate}
        \item[(1)] $D_1 \cdot D_2 \cdot \cdots \cdot D_n = 0$, and
        \item[(2)] $-(K_\cF + B) \cdot D_2 \cdot \cdots \cdot D_n >0$.
    \end{enumerate}
    Then through a general closed point $x$ of $X$, there is a rational curve $\xi$ tangent to $\cF$ such that $D_1 \cdot \xi = 0$ and 
    \[
    N\cdot \xi \leq 2n \frac{M\cdot D_2 \cdot \cdots D_n}{-K_\cF \cdot D_2 \cdot \cdots \cdot D_n}.
    \]
\end{theorem}

\subsection{Singularities in the sense of McQuillan}

For foliations of rank $1$, we need a slightly different definition of foliation singularities, originated from \cite{McQ04}, \cite{BM16}. We recall the definitions and notations introduced in \cite{CS24}.

\begin{definition}
    Let $X$ be a normal variety, and let $\cF$ be a foliation of rank 1 on $X$ such that $K_\cF$ is $\Q$-Cartier. Let $x\in X$ be a point. We may take an open neighborhood $U$ of $x$, and an index 1 cover $\sigma:U'\to U$ associate to $K_\cF$ such that $\sigma^{-1}\cF$ is generated by a vector field $\partial$. We say $\cF$ is \textit{singular in the sense of McQuillan} at $x \in X$, if there exists an embedding $U'\to M$ to a smooth variety $M$, such that there exists a lift $\Tilde{\partial}$ of $\partial$ to a vector field of $M$ such that $\Tilde{\partial}$ vanishes along $\sigma^{-1}(x)$. We denote by $\Sing^+\cF$ the locus of points $x\in X$ such that $\cF$ is singular in the sense of McQuillan.
\end{definition}

In general, we have $\Sing \cF \subset \Sing^+\cF$ by \cite[Lemma 4.1]{CS24}, provided $K_\cF$ is $\Q$-Cartier. The equality holds if $X$ admits klt singularities, see \cite[Proposition 2.32]{CS20}.

\subsection{Simple singularities}
We recall the definition of simple singularities for foliation of rank 1 on threefolds following \cite{MP13} and \cite{CS20}, which plays an important role in the foliated minimal model program of rank 1 on threefolds. 

\begin{example}[{\cite[Example III.iii.3]{MP13}}]\label{ex:simple}
    Let $Y$ be the quotient of $\C^3$ by the $\Z/2\Z$-action $(x, y, z)\mapsto (y,x,-z)$. We consider the vector field $\partial$ on $\C^3$ given by 
    \[
    \partial = \left( x \partial_x - y \partial_y\right) + \left(a(xy,z)x\partial_x - a(xy,-z)y\partial_y + c(xy,z)\partial_z  \right),
    \]
    where $a$, $c$ are formal functions in two variables and $c$ is not a unit and
    satisfies $c(xy,z) = c(xy,-z)$. Note that $\partial \mapsto -\partial $ under the group action. Therefore $\partial$ defines a foliation $\cG$ on $Y$ with an isolated singularity. Moreover, $(Y,\cG)$ is canonical. 
\end{example}

\begin{definition}[cf. {\cite[Definition 2.24]{CS20}}]
    Let $X$ be a normal threefold and let $\cF$ be a rank 1 foliation on $X$ with canonical singularities. We say that $\cF$ admits a \textit{simple singularity} at $P\in X$ if either
    \begin{enumerate}
        \item[(1)] $\cF$ is terminal at $P$ and no component of $\Sing X$ through $P$ is $\cF$-invariant; or
        \item[(2)] $(X,\cF)$ is formally isomorphic to $(Y,\cG)$ defined in Example~\ref{ex:simple} at $P$; or
        \item[(3)] $X$ is smooth at $P$.
    \end{enumerate}
    Moreover, we say that $\cF$ has simple singularities if $\cF$ admits a simple singularities at every $P\in X$. 
\end{definition}

\begin{theorem}\label{thm:simpleprop}
    Let $X$ be a normal threefold and let $\cF$ be a foliation of rank $1$ on $X$. Then the following hold:
    \begin{enumerate}
        \item[(1)] There exists a sequence of weighted blow-ups in foliation invariant centres $p: X' \to X$ such that $p^{-1}\cF$ has simple singularities.
        \item[(2)] If $\cF$ has simple singularities, then $X$ has cyclic quotient singularities. In particular, $X$ is $\Q$-factorial klt. 
    \end{enumerate}
\end{theorem}
\begin{proof}
    See \cite[III.iii.4]{MP13} and \cite[Lemma 2.26]{CS20}.
\end{proof}

\section{Normal surfaces}\label{section:normalsurfaces} In this section we recall  Mumford's intersection theory on normal surfaces, introduce the notion of surface numerical generalized foliated quadruple, and prove the cone theorem for surface numerical generalized foliated quadruples. In some of the following proofs, we freely use the MMP for foliated surfaces, following \cite{Bru02,McQ08,Bru15}.

\begin{definition}[{\cite[II.b]{Mum61}, cf. \cite{Sak84}}]

Let $S$ be a normal surface, and $f: S' \to S$ be a birational morphism from a $\Q$-factorial normal surface $S'$. Let $\bigcup E_i$ be the exceptional locus of $f$. For any $\R$-Weil divisor $D$ on $S$, we define the \emph{pullback} $f^*D$ to be 
\[
f^*D := f^{-1}_* D + \sum \alpha_i E_i,
\]
where $ f^{-1}_* D$ is the strict transform of $D$ along $f$ and $\alpha_i$ are uniquely defined by 
$$\left(f^{-1}_*D + \sum \alpha_j E_j\right)\cdot E_i = 0 \quad \text{for all } i.$$
For $\R$-Weil divisors $D_1$ and $D_2$ on $S$, we define their intersection number to be 
\[
D_1 \cdot D_2 := (f^* D_1)\cdot(f^*D_2). 
\] 
If $D_1$ is $\R$-Cartier, the definition coincide with the classical case. The definition is independent of the choice of the model $S'$.

\end{definition}

\begin{definition}[Cone of curves]

For $\R$-Weil divisors $D_1$ and $D_2$ on a normal surface $S$, we say that $D_1$ and $D_2$ are \emph{numerically equivalent in the sense of Mumford} if $D_1 \cdot C = D_2 \cdot C$ for every curve $C$ on $S$.  
We denote by $N_1^M(S)_{\mathbb{R}}$ the space of numerical equivalence classes of $\mathbb{R}$-Weil divisors on $S$ in the sense of Mumford.  
If $S$ is $\mathbb{Q}$-factorial, then $N_1^M(S)_{\mathbb{R}}$ coincides with $N_1(S)_{\mathbb{R}}$.  
Note that for any birational morphism $f \colon S' \to S$ from a $\mathbb{Q}$-factorial normal surface $S'$, there is a surjective linear map $N_1(S')_{\mathbb{R}} \to N_1^M(S)_{\mathbb{R}}$ induced by the strict transform.  
In particular, $N_1^M(S)_{\mathbb{R}}$ is a finite-dimensional $\mathbb{R}$-vector space. The Mumford's intersection theory gives a non-degenerate bilinear form on $N_1^M(S)_\R$. 
We denote by $\NE^M(S)$ the closed cone in $N_1^M(S)_{\mathbb{R}}$ generated by effective Weil divisors.

\end{definition}

\begin{definition}
    For an $\R$-Weil divisor $D$ on a normal surface $S$, we say $D$ is \textit{nef} (resp. \textit{strictly nef}) if $D \cdot C \geq 0$ (\(D \cdot C > 0\)) for any irreducible curve $C$ on $S$, and we say $D$ is \textit{numerically ample} if it is strictly nef and $D^2 >0$.
\end{definition}

\begin{lemma}\label{lemma:bnefonsurface}
    Let $S$ be a normal surface, and let $f:S'\to S$ be a birational morphism from a normal surface $S'$ to $S$. 
    \begin{enumerate}
        \item[(1)] Let $\{E_i\}$ be all the irreducible components of the exceptional set of $f$. Then the intersection matrix $(E_i \cdot E_j)$ is negative definite.
        \item[(2)] If $D'$ is a nef (resp. strictly nef, numerically ample) $\R$-divisor on $S'$, then $D = f_*D'$ is nef (resp. strictly nef, numerically ample). In particular, let $\bM$ be a b-nef b-\R-Cartier divisor on a normal surface $S$, then $\bM_S$ is nef as an $\R$-Weil divisor.
    \end{enumerate}

\end{lemma}

\begin{proof}
    (1) is the normal surface version of Grauert’s contraction criterion theorem \cite[Theorem 1.2]{Sak84}. For (2), assume that $D'$ is nef on $S'$. We may write 
    \[
    D' + E^+ - E^- = f^*D,
    \]
    where $E^+$, $E^- \geq 0$ are $f$-exceptional divisors with no common components. Since $(f^*D)\cdot E_i = 0$ for each $i$ and $D'$ is  nef, $(E^+ -E^-) \cdot E_i \leq 0$ for each $i$. In particular, $(E^+ - E^-) \cdot E^- \leq 0$. On the other hand, by (1) we have 
    \[
    0 \geq (E^+ - E^-) \cdot E^- = E^+ \cdot E^- - (E^-)^2 \geq 0.
    \]
    Hence the equality holds which implies $E^- = 0$. For any irreducible curve $C$ on $S$, let $C' = f^{-1}_*C$. Since $C'$ is not exceptional, $E^+ \cdot C'\geq 0$, then we have
    \begin{align*} 
        D\cdot C =  (f^* D ) \cdot C' = D' \cdot C' + E^+ \cdot C' \geq D' \cdot C' \geq  0.
    \end{align*}
    Therefore $D$ is nef. If $D'$ is strictly nef, then $D\cdot C \geq  D'\cdot C' >0$. If $D'$ is numerically ample, then 
    \begin{align*}
        D^2 = D' \cdot f^*D = D'\cdot (D' + E^+) = D'^2 + D'\cdot E^+ \geq D'^2 >0. 
    \end{align*}
    We may assume that $\bM$ descends to $S'$ and we may conclude.
\end{proof}

\begin{definition}[Surface num-gfqs]\label{def:NumGFQ}
    A \textit{surface numerical generalized foliated sub-quadruple} (\textit{surface sub-num-gfq} for short) $(S,\cF,B,\bM)$ consists of a normal surface $S$, a foliation $\cF$ of rank $1$, an $\R$-divisor $B$, and an b-nef b-divisor $\bM$ on $S$. A surface sub-num-gfq $(S,\cF,B,\bM)$ is called a \textit{surface num-gfq} if $B\geq 0$.

    Let $(S,\cF,B,\bM)$ be an surface num-gfq. For any prime divisor $E$ over $S$, let $f: S'\to S$ be a resolution of $S$ such that $E$ is a divisor on $S'$. Let $\cF_{S'} = f^{-1}\cF$, and we may write
    \[
    K_{\cF_{S'}} + B_{S'} + \bM_{S'} = f^*(K_\cF + B + \bM_S),
    \]
    where the pullback is defined in the sense of Mumford. We define the \textit{discrepancy} of $E$ with respect to $(S,\cF,B,\bM)$  to be  \[a_{\text{num}}(E,\cF,B,\bM) := -\mu_E B_{S'}.\] By \cite[Definition 6.4.9]{CHLX23}, this definition is independent of the choice of the resolution and if $(S,\cF,B,\bM)$ is a sub-gfq, then $a_{\text{num}}(E,\cF,B,\bM) = a(E,\cF,B,\bM)$. We say a surface num-gfq $(S,\cF,B,\bM)$ is \textit{log canonical} (\textit{lc} for short) if $a
    _{\text{num}}(E,\cF,B,\bM) \geq -\epsilon(E)$ for any prime divisor $E$ over $X$.
\end{definition}

\begin{lemma}\label{lemma:surfacemodel}
Let $S$ be a normal surface and let $\cF$ be a foliation of rank $1$ on $S$. Then there exists a birational morphism $g: S'\to S$ such that we may write 
\[
K_{\cF'}+ \Delta' = g^* K_\cF, 
\]
where the pullback is defined in the sense of Mumford, $\cF'=g^{-1}\cF$ such that $(S',\cF')$ is canonical and $\Delta' \geq 0$ is a $\Q$-divisor exceptional over $S$. Moreover, $S'$ has cyclic quotient singularities.
\end{lemma}

\begin{proof}
    Take a resolution $f: X \to S$ of $S$ such that $\mathcal{G} = f^{-1}\cF$ has reduced singularities. We may write
    $
    K_{\mathcal{G}} + \Delta = f^*K_\cF.
    $
    By \cite[Corollary 2.26]{Spi20}, we may run a $K_{\mathcal{G}}$-MMP over $S$, and assume the MMP terminates on $S'$. Let $g:S' \to S$ be the induced morphism and $\alpha:X\to S'$ be the composition of each step of the MMP. Let $\cF' = g^{-1}\cF$ and  $\Delta' = \alpha_*\Delta$, then we have $K_{\cF'} + \Delta' = g^*K_\cF$. Note that $\Delta' = -K_{\cF'} + g^*K_\cF$ is exceptional over $S$ and anti-nef over $S$, we have $\Delta' \geq 0$ by Lemma~\ref{lemma:bnefonsurface}. Moreover, $(S',\cF')$ is canonical as $(X,\cG)$, with reduced singularities, is canonical.
\end{proof}

\begin{proposition}\label{prop:surfnum}
    Let $S$ be a normal projective surface and $(S,\cF,B,\bM)$ be a surface num-gfq. Let $C \subset S$ be a irreducible curve not contained in $\Supp B$. Then:
    \begin{enumerate}
        \item[(1)] If $C$ is not $\cF$-invariant, then $(K_\cF + B + C + \bM_S)\cdot C \geq 0 $.
        \item[(2)] If $C$ is $\cF$-invariant, and $(K_\cF + B + \bM_S) \cdot C <0 $, then $C$ is a rational curve, and $0 <-(K_\cF + B + \bM_S)\cdot C \leq 2$.
    \end{enumerate}
\end{proposition}

\begin{proof}
For (1), by Lemma~\ref{lemma:bnefonsurface}, we have $\bM_S \cdot C \geq 0$. Since $C$ is not contained in $\Supp B$, we also have $B \cdot C \geq 0$.  
Thus, it remains to show that $(K_\cF + C) \cdot C \geq 0$. The proof is similar to \cite[Proposition 3.4]{Spi20}, but we present a complete proof for the reader’s convenience.
Let $\mu:Y \to S$ be a log resolution of $(X,C)$ and let $\cG = \mu^{-1}\cF$. Possibly replacing $\mu$ with a higher resolution we may assume that $\cG$ has reduced singularities. Let $C_Y = \mu^{-1}_* C$. By \cite[Corollary 2.2]{Spi20}, we may run a $(K_\cG + C_Y)$-MMP over $S$, which terminates on $\pi:X \to S$. 

We claim that this MMP only contracts curves $E$ which are disjoint from the strict transform of $C$. We apply induction on the number of steps of this MMP. Let $Y \to Y'$ be an intermediate step of the MMP, let $\cG'$ be the transformed foliation on $Y'$, and let $C_{Y'}$ be the strict transform of $C$ on $Y'$. Note that by induction $Y'$ is smooth in a neighborhood of $C_{Y'}$, then for any curve $E$ intersects $C_{Y'}$ we have $ C_{Y'}\cdot E\geq 1$. On the other hand, the relative MMP only contracted curves $E$ such that $(K_{\cG'} + C_{Y'})\cdot E <0$ and $0>K_{\cG'}\cdot E \geq -1$, then the contracted curve cannot intersect with $C_{Y'}$.

Let $\cF_X = \pi^{-1}\cF$ and $C_X = \pi^{-1}_* C$. Since $K_{\cF_X} + C_X$ is nef over $S$, by Lemma~\ref{lemma:bnefonsurface}, we may write 
\[
K_{\cF_X} + C_X + \Gamma = \pi^*(K_\cF + C),
\] 
where $\Gamma \geq 0$ is exceptional over $S$. Therefore
\[
(K_\cF + C)\cdot C  = \pi^*(K_\cF + C) \cdot C_X = (K_{\cF_X} + C_X + \Gamma) \cdot C_X \geq 0,
\]
where the last inequality holds by \cite[Proposition-Definition 3.7]{CS24} and the fact that $X$ is smooth in a neighborhood of $C_X$.

For (2), since $B \cdot C \geq 0$ and $\bM_S \cdot C \geq 0$, we have $K_\cF \cdot C <0$. By Lemma~\ref{lemma:surfacemodel}, there exists a birational morphism $g:S'\to S$ such that $K_{\cF'} + \Delta' = g^*K_\cF$ where $\Delta' \geq 0$ is exceptional over $S$ and $\cF' = g^{-1}\cF$ has canonical singularities. Let $C'$ be the strict transform of $C$ to $S'$. then $0 > K_{\cF}\cdot C = g^*K_\cF \cdot C' = (K_{\cF'}+ \Delta')\cdot C'\geq K_{\cF'} \cdot C' $. Then $C'$ is $K_{\cF'}$-negative, then we have
\[
0 > K_{\cF'} \cdot C' = 2p_a(C')-2 + Z(\cF',C').
\]
We refer the reader to \cite[\S 2]{Bru02} for the definition of $Z(\cF',C')$. Since $\cF'$ is canonical, we have $Z(\cF',C') \geq 0$. Therefore $C'$ is a rational curve and so is $C$. Moreover, $K_{\cF'}\cdot C'\geq -2 $, and $(K_\cF + B + \bM_S) \cdot C \geq K_\cF \cdot C \geq K_{\cF'} \cdot C' \geq -2$.

\end{proof}

In the end of this section we give the cone theorem for the surface num-gfq, which generalizes \cite[Theorem 6.3]{Spi20}.

\begin{proposition}\label{prop:conethmnumgfq}
    Let $S$ be a normal projective surface and $(S,\cF,B,\bM)$ be a rank $1$ surface num-gfq. Let $f:X \to S$ be the minimal resolution of $S$. Let $Z_{-\infty}$ be the subcone of $\NE^M(S)$ spanned by those prime divisors $D\subset \Supp B$ such that $\mu_D B >\epsilon(D)$. Let $\{R_j\}_{j\in \Lambda}$ be the set of all \( (K_\cF + B + \bM_S) \)-negative extremal rays of $\NE^M(S)$ not contained in $Z_{-\infty}$. Then the following hold:
    \begin{enumerate}
    \item[(1)] We may write 
    \[
        \NE^M (S) = \NE^M(S)_{K_\cF + B + \bM_S\geq 0} + Z_{-\infty} + \sum_{j \in \Lambda} R_j.
    \]
    \item[(2)] For each $j \in \Lambda$, $R_j$ is generated by a rational curve $C_j$ which is tangent to $\cF$, and 
    \[
    0 < -(K_\cF + B + \bM_S) \cdot C_j \leq 2.
    \]
    \item[(3)] For each ample $\R$-divisor $A'$ on $X$, let $A = f_* A'$. Then the set 
    \[
    \Lambda_A :=\{\, j \in \Lambda \mid R_j \subset \NE^M(S)_{K_\cF + B + \bM_S +A < 0} \, \}
    \]
    is finite.
    \item[(4)] Every $(K_\cF + B + \bM_S)$-negative extremal ray is generated by an irreducible curve.
    \end{enumerate}
\end{proposition}

\begin{proof} Let $R$ be an $(K_\cF+ B + \bM_S)$-negative exposed ray of $\NE^M(S)$.

    \medskip

    \textit{Step 1}.
    In this step we show that $R$ is generated by an irreducible curve $C$. 
    
    Let $H_R$ be a supporting function to $R$. Then $H_R$,  as an $\R$-Weil divisor, is nef and not numerically trivial. We may write $H_R = K_\cF + B + \bM_S + D$ for some numerically ample $\R$-divisor $D$ on $S$, then $f^*H_R$ is nef, so $H_R^2 = (f^*H_R)^2 \geq 0$.

    First we consider the case when $H_R^2 = 0$. There exists a unique integer $\nu\in\{0,1\}$ satisfying $H_R^\nu \cdot D^{2-\nu} >H_R^{\nu + 1} \cdot D^{1-\nu} = 0$. If $\nu = 0$, set $D_1 = H_R$ and $D_2 = D$. Otherwise set $D_1 = D_2 = H_R$. Then $D_1\cdot D_2 = 0$, $-(K_\cF + B + \bM_S) \cdot D_2 >0$. Since $B \geq 0 $, $\bM_S $ is pseudo-effective and $D_2$ is nef, we have $(B +\bM_S)\cdot D_2 \geq 0$, then $-K_\cF \cdot D_2 > 0$. Set $M = H_R + D = 2H_R-(K_\cF + B + \bM_S)$, then $M$ is nef on $S$.

    Let $D_1'$, $D_2'$, $M'$ be the pullback of $D_1$, $D_2$, $M$ to $X$ respectively. Then $D_1'$, $D_2'$, $M'$ is nef on $X$. Let $\cG = f^{-1}\cF$, then $- K_\cG \cdot D_2' = - K_\cF \cdot D_2 >0$. By Theorem~\ref{thm:B&B} ({\cite[Corollary 2.28]{Spi20}}), there exists a rational curve $C'$ through a general point of $X$, such that $C'$ is tangent to $\cG$, and $D_1' \cdot C' = 0$. Since $C'$ passes through a general point of $X$, let $C = f(C')$, then $C$ is a rational curve on $S$ which tangent to $\cF$. Note that $(K_\cF + B + \bM_S + D) \cdot C = H_R \cdot C = D_1' \cdot C' = 0$, then $R$ is generated by $C$. Also we have $(K_\cF + B + \bM_S) \cdot C < 0$. By Proposition~\ref{prop:surfnum} (2), $0<-(K_\cF + B + \bM_X)\cdot C \leq 2$.

    Now we assume that $H_R^2 > 0$, then $f^*H_R$ is nef and big. We may write $f^*H_R \sim_\R \Gamma' + E'$, where $\Gamma'$ is ample and $E' \geq 0$. Note that $f$ induces a surjective map $f_*: \NE(X)\to \NE^M(S)$. By Lemma~\ref{lemma:ExtRayPullback}, there exists an extremal ray $R'$ with $f_*(R') = R$. We have $(\Gamma' + E') \cdot R' = f^*H_R\cdot R' = H_R \cdot R = 0$, then $E' \cdot R' <0$. Therefore there exists an irreducible curve $C'$ on $X$, which is contained in $\Supp E'$ and generates $R'$. Since $f_*R' = R \neq 0$ in $\NE^M(S)$, $C'$ is not $f$-exceptional. Thus $R$ is generated by $C = f_*C'$. Let $\Gamma = f_*\Gamma'$ and $E = f_* E'$. Since $\Gamma'$ is ample, then we may write $\Gamma' + \Xi = f^*\Gamma$ for some exceptional $\Xi \geq 0$. Since $C'$ is not $f$-exceptional, $\Xi \cdot C'\geq 0$. Then $\Gamma \cdot C = f^* \Gamma \cdot C' = \Gamma' \cdot C' + \Xi \cdot C' >0 $, and since $(\Gamma + E) \cdot C = H_R\cdot C = 0$, we have $E \cdot C<0$, $C^2 <0$. 
    
    \medskip 

    \textit{Step 2}.
    Now we assume that $R$ is not contained in $Z_{-\infty}$. Let $C$ be an irreducible curve which generates $R$. We prove that $0<-(K_\cF + B + \bM_S) \cdot C \leq 2$. By Step 1, we may assume $H_R^2 >0$ and $C^2 <0$. Then $(S,\cF,B,\bM)$ is lc at the generic point of $C$. If $C$ is not $\cF$-invariant, then since $\mu_C B \leq 1$, we may write $B + tC = \Delta + C$, where $t \geq 0$, $\Delta$ is effective and $C$ is not contained in $\Supp \Delta$. Then by Proposition~\ref{prop:surfnum} (1), $$(K_\cF + B + \bM_S + tC)\cdot C = (K_\cF + \Delta + C + \bM_S)\cdot C \geq 0,$$
    which contradicts with the fact that $C^2 <0$ and that $R$ is $(K_\cF + B + \bM_S)$-negative. Then $C$ is $\cF$-invariant and $C$ is not contained in $\Supp B$. By Proposition~\ref{prop:surfnum} (2), $0<-(K_\cF + B + \bM_S) \cdot C \leq 2$.

    \medskip 

    \textit{Step 3}. In this step we prove (3). Assume $\Lambda_A$ is not a finite set. Let $K = f^*(K_\cF + B + \bM_S)$, and let $A_1',\dots ,A_{\rho}'$ be ample Cartier divisors on $X$ such that $A_1',\dots ,A_{\rho}'$ form a basis for $N^1(X)_\R = N_1(X)_\R$. Here $\rho = \rho(X)$. Let $H' =  \sum_{k=1}^{\rho}A_k'$ and $H = f_*H'$. There exists a sufficiently small $\epsilon >  0 $ such that $A'-\epsilon H' $ is ample. By Lemma~\ref{lemma:bnefonsurface}, we have $\Lambda_A \subset \Lambda_{\epsilon H}$. Thus we may assume that $A' = \epsilon H' = \epsilon \sum_{k=1}^{\rho}A_k'$.

    By \cite[Theorem 18.6]{Roc70}, we may replace $\Lambda_A$ by a countable infinite subset and assume that for each $j\in \Lambda_A$, $R_j$ is a $(K_\cF + B + \bM_S + A)$-negative exposed ray. By Step 1 and 2 above, $R_j$ is generated by $C_j$ with $(K_\cF + B + \bM)\cdot C_j \in [-2,0)$. Let $C_j' = f^{-1}_*C_j $. Then $K \cdot C_j' = (K_\cF + B + \bM_S)\cdot C_j \in [-2,0)$.  For each $j\in \Lambda_A$, there exists a numerically ample $L_j$ on $S$, such that $H_j = (K_\cF + B +\bM_S + A)+ L_j$ is a supporting function of $R_j$. Then $H_j \cdot C_j = 0$, and we have
    \begin{align*}
        0 &= H_j \cdot C_j = ((K_\cF + B +\bM_S + A)+ L_j)\cdot C_j \\
        &=f^*(K_\cF + B + \bM_S + A)\cdot C_j' + L_j \cdot C_j \\
        &=(K + f^*A)\cdot C_j' + L_j \cdot C_j \geq (K + f^*A)\cdot C_j' \\
        &= K \cdot C_j' + (A' + \Theta)\cdot C_j'\geq K \cdot C_j' + A' \cdot C_j' \\
        &\geq  -2 + \epsilon\sum_{k=1}^{\rho}(A_k' \cdot C_j').
    \end{align*}
    where $f^*A = A' + \Theta$, $\Theta\geq 0$ is exceptional over $S$. Therefore, $(A_k '\cdot C_j')\in(0, \frac{2}{\epsilon}]$ for each $k$. Since for each $k$, $A_k'$ are ample Cartier divisors, $(A_k '\cdot C_j')\in \Z$. After replacing $\Lambda_A$ with a countable infinite subset, we may assume that the values of $(A_k '\cdot C_j')$ are independent of $j$. Thus for every $j, j'\in\Lambda_A$, $C_{j}'\equiv C_{j'}'$, then $C_j \equiv C_{j'}$, which leads to a contradiction. Thus (3) follows.

    \medskip

    \textit{Step 4.} In this step we finish the proof. It suffices to show that for every $j \in \Lambda$, $R_j$ is exposed. If $R_j$ is extremal but not exposed, we may assume $R_j$ is $(K_\cF + B + \bM_S + A)$-negative, where $A$ is the pushforward of some sufficiently small ample $\R$-divisor on $X$. By \cite[Theorem 18.6]{Roc70}, there exist exposed rays $\{R_{j,i}\}_{i=1}^\infty$ such that $R_j = \lim_i R_{j,i}$. By (3), $\Lambda_A$ is finite, then we may assume that $R_{j,i}\in Z_{-\infty}$ for all $i$. Since $Z_{-\infty}$ is closed, we have $R_j\subset Z_{-\infty}$, which is a contradiction. Thus by Step 1 and 2, (1) and (2) follow. (4) follows from Step 1 and the fact that $Z_{-\infty}$ is spanned by finitely many irreducible curves. 
\end{proof}

\section{Cone theorem}

\subsection{Invariant curves}

We give a proposition that shows the b-divisor has non-negative degree along invariant curves, provided that the generalized foliated quadruple is log canonical. This is a key step in the proof of the following cone theorem and in the construction of the MMP for generalized foliated quadruples on threefolds of rank $1$.

\begin{proposition}\label{proposition:bdivoncurve}

    Let $X$ be a normal variety and $(X, \cF, B, \bM)$ be a rank 1 gfq. Assume that $K_\cF$ and $B$ are $\R$-Cartier. Let $C $ be an $\cF$-invariant curve such that $(X, \cF, B, \bM)$ is lc at the generic point of $C$. Then $\bM_X \cdot C \geq 0$.
\end{proposition}

\begin{proof}

\textit{Step 1.} It is clear that $\bM_X$ is $\R$-Cartier. We claim that there exists a birational morphism $f: {X'} \to X$ such that the following hold:
    \begin{enumerate}
        \item[(1)] $\bM_{X'} + F = f^*\bM_X$, where $F\geq 0$ is an $f$-exceptional $\R$-divisor on $X'$, and 
        \item[(2)] there exists a prime divisor $E$ on $X'$ such that the center of $E$ on $X$ is $C$, and $E$ is not contained in $\Supp F$.
    \end{enumerate}
    First we show that it suffices to prove the claim above. Let $g: X'' \to X'$ be a birational morphism such that $X''$ is smooth and $\bM $ descends to $X''$. We may write $\bM_{X''} + \Theta = g^*f^*\bM_X$. By the negativity lemma, we have $\Theta \geq 0$. Note that $F = g_*\Theta$. Let $E'$ be the strict transform of $E$ on $X''$, then $E'$ is not contained in $\Supp \Theta$. Take a general curve $C' \subset E'$ which dominates $C$. Then $\Theta \cdot C' \geq 0$ as $C'$ is not contained in $\Supp \Theta$. Let $d$ be the degree of the map $C' \to C$. Thus
    \begin{align*}
        \bM_{X} \cdot C = \frac{1}{d} (g^*f^*\bM_X \cdot C') = \frac{1}{d} (\bM_{X''}\cdot C' + \Theta\cdot C' )  \geq 0.
    \end{align*}

\medskip

\textit{Step 2.} In this step we reduce to the case when $K_\cF$ is Cartier. Note that it suffices to construct a morphism required in the claim of Step 1 over a neighborhood of a general point on $C$. After replacing $X$ with an open neighborhood of a general point of $C$, we may assume that there exists an index one cover $\sigma: Y \to X$ associated to $K_\cF$. Let $\cG = \sigma^{-1}\cF$, $B_Y = \sigma^*B$ and $\bM^Y = \sigma^* \bM$. Then $K_{\cG} = \sigma^*K_\cF$ is Cartier and by Lemma~\ref{lemma:finitepullbackMX}, $\bM^Y_Y = \sigma^* \bM_X$. Then we have 
\[
K_\cG + B_Y + \bM^Y_Y = \sigma^*(K_\cF + B + \bM_X).
\]
Let $\tilde{C}$ be an irreducible curve on $Y$ such that $\sigma(\tilde{C}) = C$. By Lemma~\ref{lemma:GFQfinite}, $(Y, \cG,B_Y,\bM^Y)$ is lc at the generic point of $\tilde{C}$. By \cite[Lemma 4.2]{CS24}, $\tilde{C}$ is $\cG$-invariant.

Assume that there exists a birational morphism $g: Y' \to Y$ such that $\bM^Y_{Y'} + F_Y = g^*\bM^Y_Y$, where $F_Y \geq 0$ is $g$-exceptional, and that there exists a prime divisor $E_Y$ on $Y'$ whose center on $Y$ is $\tilde{C}$ and which is not contained in $\Supp F_Y$. 
By \cite[Lemma 2.22]{Kol13}, there exists a birational morphism $f: X' \to X$ and a finite morphism $\sigma': Y' \to X'$ such that $E = \sigma'(E_Y)$ is a prime divisor exceptional over $X$ and is centered on $C$. 
After replacing $Y',X'$ with higher models, and replacing $E_Y$ with the strict transform, we may assume that $\bM$ descends to $X'$ and $X'$ is smooth. We may write $\bM_{X'} + F = f^* \bM_{X'}$, where $F \geq 0$ is $f$-exceptional. Note that $\bM^Y_{Y'} = \sigma'^* \bM_{X'}$. Therefore 
\[
\bM^Y_{Y'} + F_Y = g^* \bM^Y_Y = g^*\sigma^* \bM_X = \sigma'^*f^* \bM_X  = \sigma'^*(\bM_{X'} + F) = \bM^Y_{Y'} + \sigma'^* F,
\]
then we have $F_Y = \sigma'^*F$. If $E \subset \Supp F$, then $E_Y \subset \Supp F_Y$, which contradicts to the assumption of $E_Y$. Thus $E$ is not contained in $\Supp F$, and the claim in Step 1 holds for $f: X' \to X$ and $E$. Therefore, after replacing $(X,\cF,B,\bM),C$ with $(Y,\cG,B_Y,\bM^Y),\tilde{C}$, we may assume that $K_\cF$ is Cartier.

\medskip

\textit{Step 3.} In this step we prove the claim in Step 1 when $\cF$ is canonical at the generic point of $C$. Let $\phi: W \to X$ be the blow-up of $X$ along $C$ and let $\cF_W = \phi^{-1}\cF$. Since $C$ is $\cF$-invariant, by \cite[Lemma 1.1.3]{BM16} there exists a $\phi$-exceptional divisor $E_1 \geq 0$ with $K_{\cF_W} +E_1 = \phi^* K_\cF$. Let $B_W = \phi^{-1}_* B$, then $B_W + E_2 = \phi^* B$ for some $\phi$-exceptional $\R$-divisor $E_2 \geq 0$. Furthermore, by the negativity lemma, there exists a $\phi$-exceptional $\R$-divisor $E_3 \geq 0$ such that $\bM_W + E_3 = \phi^* \bM_X$. Therefore if we let $\Gamma = E_1 + E_2 + E_3$, then $\Gamma \geq 0$ and $K_{\cF_W} + B_W + \bM_{W} + \Gamma = \phi^*(K_\cF + B + \bM_X)$. Let $D$ be an irreducible component of $\Exc (\phi)$ which dominates $C$. Since $\cF$ is canonical at the generic point of $C$, we have $D$ is $\cF_W$-invariant by \cite[Corollary III.i.4]{MP13}. The lc condition implies $\mu_D\Gamma \leq \epsilon(D) = 0$, then $D$ is not contained in $\Supp \Gamma$ therefore also not contained in $\Supp E_3$. Thus the claim in Step 1 holds for $\phi: Y \to X $ and $D$. Note that $a(D,\cF) = -\mu_D E_1 = 0$.

\medskip

\textit{Step 4.} Now we may assume that $\cF$ is not canonical at the generic point of $C$. Since $(X,\cF)$ is also lc at the generic point of $C$, there exists a prime divisor $G$ over $X$ centered on $C$ such that $\epsilon (G) = 1$ and $a(G,\cF) <0$. Let $\psi: V \to X$ be a birational morphism such that $G$ is a prime divisor on $V$. Since $K_\cF$ is Cartier, $a(G,\cF)$ is an integer, then $a(G,\cF) = -1$ (see also \cite[Fact III.i.3]{MP13}). Let $\cF_V = \psi^{-1}\cF$, $B_V = \phi^{-1}_* B$. Similarly to Step 3, we may write $K_{\cF_V} + F_1 = \psi^*K_\cF$, $B_V + F_2 = \psi^* B$, and $\bM_V + F_3 = \psi^* \bM_X$, where $F_1,F_2,F_3$ are exceptional over $X$, $F_2,F_3 \geq 0$, and $\mu_G F_1 = - a(G,\cF) = 1$. Since $(X,\cF,B,\bM)$ is lc at the generic point of $C$, we have
\[
1 = \mu_G(F_1) \leq \mu_G(F_1 + F_2 + F_3) = -a(G,\cF,B,\bM) \leq \epsilon (G) = 1.
\]
Thus $G$ is not contained in $\Supp F_3$, and the claim in Step 1 holds for $\psi: V \to X$ and $G$.
\end{proof}
\begin{remark}\label{rmk:Clccenter}
    The proof above, combined with Lemma~\ref{lemma:GFQfinite}, also shows that in the setting of Proposition~\ref{proposition:bdivoncurve}, the curve $C$ is a log canonical center of $ (X,\cF) $. 
\end{remark}

\begin{corollary}\label{coro:singinvcurve}
    Let $X$ be a normal variety and let $(X, \cF, B, \bM)$ be a rank 1 gfq. Assume that $K_\cF$ and $B$ are $\R$-Cartier. Let $C \subset \Sing ^+ \cF$ be a curve and suppose that $(X, \cF, B, \bM)$ is lc at the generic point of $C$. Then $(K_\cF + B + \bM_X) \cdot C \geq 0$.
\end{corollary}

\begin{proof}
    Since $(X,\cF,B)$ is also an lc foliated triple, by \cite[Lemma 4.7]{CS24} we have $(K_\cF + B ) \cdot C\geq 0$. It suffices to prove that $\bM_X \cdot C \geq 0$. It suffices to construct a birational morphism as in Step 1 of the proof of Proposition~\ref{proposition:bdivoncurve}, which is a local statement around a general point of $C$. As in Step 2 of the proof of Proposition~\ref{proposition:bdivoncurve}, we may also replace $X$ by a index one cover associated to $K_\cF$ over a neighborhood of a general point of $C$ and assume that $K_\cF$ is Cartier and $\cF$ is generated by a vector field $\partial$. By the definition of $\Sing^+ \cF$, $C$ is $\cF$-invariant, and we may conclude by the same proof of Proposition~\ref{proposition:bdivoncurve}.
\end{proof}

\subsection{Non-invariant curves}

Let $X$ be a $\Q$-factorial normal variety and $(X, \cF, B, \bM)$ be a gfq of rank 1.  
As in \cite{CS24}, for a non-invariant curve $C$, after replacing $X$ by a suitable birational model $X'$, we aim to find an $\cF$-invariant surface containing $C$.  The following lemma is a variant of \cite[Lemma 4.6]{CS24} with the $\Q$-factoriality condition.

\begin{lemma}\label{lemma:findsurface,Q-fac}
    Let $X$ be a normal $\Q$-factorial variety, let $\cF$ be a foliation of rank $1$ on $X$. Let $C$ be a curve on $X$ which is not $\cF$-invariant and which is not contained in $\Sing^+ \cF$. Then there exists a birational morphism $p: X' \to X$ such that the following hold:
    \begin{enumerate}
        \item[(1)] Let $\cF' = p^{-1} \cF$, then $K_{\cF'} + E = p^*K_\cF$, where $E \geq 0$ is a $p$-exceptional divisor;
        \item[(2)] $p^{-1}$ is an isomorphism at the generic point of $C$, and let $C'$ be the strict transform of $C$ to $X'$, then $\cF'$ is terminal at all points of $C'$; 
        \item[(3)] after possibly replacing $X$ by an analytic neighborhood of $C$, there exists a $\cF'$-invariant surface $\Gamma$ containing $C'$; and
        \item[(4)] $X'$ is $\Q$-factorial. 
    \end{enumerate}
\end{lemma}

\begin{proof}
    We mostly follow the proof of \cite[Lemma 4.6]{CS24} and present the full argument for the reader’s convenience. 
    By \cite[Lemma 2.9]{CS20}, $\cF$ is terminal at general points on $C$. 
    Let $P_1,\dots, P_k\in C$ be all the closed points where $\cF$ is not terminal. Let $H$ be a sufficiently general ample divisor on $X$ such that $P_i \notin \Supp H$ for $1 \leq i \leq k$. Let $M$ be the Cartier index of $K_\cF$ and we may assume that $\cO_X(M K_\cF) |_{X\setminus H} \cong \cO_{X\setminus H}$. Then we may find a finite Galois morphism $\sigma:Y \to X$ with Galois group $G$, such that $\sigma $ is quasi-\'etale outside $H$ and $\sigma^*K_\cF$ is Cartier on $Y$. Let $\cG = \sigma^{-1}\cF$, $\Tilde{C} = \sigma^{-1}(C)$, and $Z = \sigma^{-1}(\{P_1,\dots ,P_k\})$. Then $K_\cG$ is Cartier on $Y$ near each closed point $Q \in Z$. By Lemma~\ref{lemma:GFQfinite}, $\cG$ is terminal at every closed points $ Q' \in \Tilde{C} \setminus Z$.

    By the desingularization algorithm in \cite{BM97}, after possibly replacing $X$ by an analytic neighborhood of $C$, there exists a birational morphism $\alpha: \overline{Y}\to Y$, obtained as a sequence of $G$-equivariant blow-ups along $\cG$-invariant centers, such that $\alpha^{-1}$ is an isomorphism near each point of $\tilde{C} \setminus Z$, and if we let $\overline{C}$ to be the strict transform of $\Tilde{C}$ on $\overline{Y}$, $\overline{Y}$ is smooth at each closed point contained in $\overline{C} \cap \alpha^{-1}(Z)$. By \cite[Proposition 1.2.4]{BM16}, there exists a birational morphism $\beta: Y' \to \overline Y$, obtained as a sequence of $G$-equivariant blow-ups along the $\cG$-invariant closed points, such that $\beta^{-1}$ is an isomorphism near each point of $\overline{C}\setminus \alpha^{-1}(Z)$, and the strict transform $\Tilde{C}'$ of $\tilde{C}$ on $Y'$ is disjoint from $\Sing^+\beta^{-1}\alpha^{-1}\cG$. Let $q = \alpha \circ \beta:Y' \to Y$. Note that each step in the sequence of blow-ups that constructs $\beta$ is a blow-up along a closed point over $\overline{C} \cap \alpha^{-1}(Z)$, so $Y'$ is smooth at each point in $\tilde{C}' \cap q^{-1}(Z)$. Let $\cG' = q^{-1}\cG$, then by \cite[Lemma 1.1.3]{BM16}, $K_{\cG'} + E' = q^*K_\cG$ for some $q$-exceptional $E'\geq 0$.

    In the construction above, we see that $q$ is $G$-equivariant. Therefore if $X' = Y'/G$, then there exists a birational morphism $p:X' \to X$ induced by quotient of $q$ by $G$. Let $r: Y' \to X'$ be the quotient map. It follows that $K_{\cF'}+ E = p^*K_\cF$ where $\cF' = p^{-1}\cF$ and $E\geq 0$ is $p$-exceptional by \cite[Lemma 3.4]{Dru21} (see also \cite[Proposition 2.2]{CS20}). Thus (1) follows. 

    Note that $p^{-1}$ is an isomorphism near each point $P\in  C \setminus \{P_1, \dots, P_k\}$. Then $X'$ is $\Q$-factorial near $p^{-1}P$. Let $C'$ be the strict transform of $C$ on $X'$. For any $Q\in C'$ such that $p(Q)\in \{P_1, \dots, P_k\}$, since $Y'$ is smooth at each point of $r^{-1}(Q)$, $X'$ is $\Q$-factorial near $Q$ (see e.g. \cite[Theorem 3.8.1]{Ben93}). Then by \cite[Théorème 6.1]{BGS11}, after possibly shrinking $X$ near $C$, we have $X'$ is $\Q$-factorial, (4) follows. 
    
    Since $\Tilde{C}'$ is disjoint from $\Sing^+\cG'$, by \cite[Lemma 2.9]{CS20}, $\cG'$ is terminal at all closed points on $\Tilde{C}'$. By \cite[Corollary III.i.4]{MP13}, each exceptional divisor over $Y'$ centered on some closed point on $\Tilde{C}'$ is foliated invariant. Thus each exceptional divisor over $X'$ centered on some closed point on $C'$ is foliated invariant. By the computation in Lemma~\ref{lemma:GFQfinite}, $\cF'$ is terminal at all closed points of $C'$, and (2) follows.

    By \cite[Proposition 1.2.4]{BM16}, there exists a $\cG'$-invariant surface $S$ containing $\Tilde{C}'$. We may take $\Gamma = r(S)$ then (3) follows. 
\end{proof}

\begin{lemma}\label{lemma:noninvcurveininvsurf}
    Let $X$ be a normal $\Q$-factorial threefold and let $(X,\cF,B,\bM)$ be a rank 1 gfq on $X$. Let $C$ be a curve on $X$ which is not $\cF$-invariant and is not contained in $\Sing^+ \cF$, and let $S$ be an $\cF$-invariant surface such that $S$ is not contained in $\Supp B$ and $C \subset S$. Let $\nu: S^\nu \to S$ be the normalization map and let $B_S$ be an $\R$-divisor on $S^\nu$ such that
    \[
    (K_\cF + B + \bM_X)|_{S^\nu} = K_{\cF_S} + B_S + \bM^S_{S^\nu},
    \]
    where $\cF_S$ is the restricted foliation on $S^\nu$ and $\bM^S = \bM|_{S^\nu}$ is the restricted b-divisor. Then if $(X,\cF,B,\bM)$ is lc at the generic point of $C$, then $(S^\nu,\cF_S,B_S,\bM^S)$ is lc at the generic point of $\nu^{-1}(C)$.
\end{lemma}

\begin{proof}
    It is a local statement at a general point $x$ of $C$, so by Lemma~\ref{lemma:GFQfinite} we may replace $X$ by an index one cover and assume $K_\cF$ is Cartier. We may also assume that $C$ does not intersect with $\Sing^+ \cF$. Then by \cite[Lemma 1.2.1]{BM16}, there exists an analytic neighborhood $U$ of $x$ and a smooth morphism $q:U \to V$ such that $\cF|_U$ is induced by $q$. After replacing $X$ with $U$, we may assume $\cF$ is algebraically integrable. By \cite[Theorem 6.6.1]{CHLX23}, $(S^\nu,\cF_S,B_S,\bM_S)$ is lc at the generic point of $\nu^{-1}(C)$. 
\end{proof}

\begin{lemma}\label{lemma:noninv}
    Let $X$ be a normal $\Q$-factorial threefold and $(X, \cF, B, \bM)$ be a rank 1 gfq on $X$. Let $C$ be a curve on $X$ which is not $\cF$-invariant and is not contained in $\Sing^+ \cF$. Assume that $(X,\cF,B,\bM)$ is lc at the generic point of $C$. Suppose that there exists an effective $\R$-divisor $T$ on $X$ satisfying: 
    \begin{itemize}
        \item $T\cdot C < 0$, 
        \item there exists an irreducible component $T_0$ of $\Supp T$ such that $C \subset T_0$, and $C$ is not contained in any other component of $\Supp T$, and
        \item $T_0$ is not $\cF$-invariant.
    \end{itemize}
    Then $(K_\cF + B +\bM_X) \cdot C\geq 0$. 
\end{lemma}

\begin{proof}
    Assume that $(K_\cF + B + \bM_X) \cdot C< 0$. By Lemma~\ref{lemma:findsurface,Q-fac}, after possibly replacing $X$ by an analytic neighborhood of $C$, there exists a birational morphism $p: X' \to X$ which satisfies the properties stated in Lemma~\ref{lemma:findsurface,Q-fac}.
    We may write 
    \[
    K_{\cF'} + B' + \bM_{X'} = p^*(K_{\cF} + B + \bM_X),
    \]
    where $\cF' = p^{-1}\cF$. By Lemma~\ref{lemma:findsurface,Q-fac} (1) and the negativity lemma, we have $B' \geq 0$. Let $C'$ be the strict transform of $C$ on $X'$. Let $T'= p^*T$ and let $T_0'$ be the strict transform of $T_0$ on $X'$. Note that $X'$ is $\Q$-factorial. Since $p^{-1}$ is an isomorphism at the generic point of $C$, we have that $T'$, $T_0'$ and $C'$ also satisfy the conditions stated in this lemma. After possibly replacing $X,\cF,B,C , T$ and $T_0$ by $X',\cF',B',C',T'$ and $T_0'$, respectively, we may assume $\cF$ is terminal at all points of $C$ and $C$ is contained in an $\cF$-invariant surface $S$. 
    
    Since $(X,\cF,B,\bM)$ is lc at the generic point of $C$, we have $S$ is not contained in $\Supp B$. Let $S^\nu \to S$ be the normalization of $S$. By Proposition~\ref{prop:subadj,primediv}, there exists $B_S\geq 0$ on $S^\nu$ with
    \[
    K_{\cF_S}+B_S + \bM^S_{S^\nu} := (K_{\cF} + B + \bM_X)|_{S^\nu},
    \]
    where $\cF_S$ is the restricted foliation on $S^\nu$ and $\bM^S = \bM|_{S^\nu}$ is the restricted b-divisor. 
    
    By Lemma~\ref{lemma:noninvcurveininvsurf}, $(S^\nu,\cF_S,B_S,\bM_S)$ is lc at the generic point of $C$, or equivalently, $\mu_C B_S \leq 1$. Here we consider $C$ as a curve on $S^\nu$. Then there exists $t \geq 0$ such that $B_S + tC = \Delta + C$, where $\Delta \geq 0$ and $C $ is not contained in $\Supp \Delta$. Note that $S$ is $\cF$-invariant and $C \subset S$. By the assumption of $T$, $S$ is not contained in $\Supp T$. In particular, $T|_{S^\nu}$ is effective.  Since $ T|_{S^\nu} \cdot C <0$ on $S^\nu$, we have $\mu_C(T|_{S^\nu})>0$ and $C^2<0$. Thus
    \begin{align*}
        (&K_{\cF_S} + \Delta + C+ \bM_{S^\nu}^S)\cdot C  \\ =  (&K_{\cF_S} + B_S +tC+ \bM_{S^\nu}^S)\cdot C \\ =(&K_\cF + B + \bM_X) \cdot C+ tC^2 <0,
    \end{align*}
    which contradicts to (1) of Proposition~\ref{prop:surfnum}.
\end{proof}

\subsection{Cone theorem} We give a lemma that provides an estimate for the length of extremal rays in the cone theorem. Using foliation adjunction, we may obtain a more precise bound compared to the classical case.

\begin{lemma}\label{lemma:step1}
Let $X$ be a normal projective $\Q$-factorial variety and $(X,\cF,B,\bM)$ be a rank 1 gfq on $X$. Let $C$ be a $(K_\cF + B + \bM_X) $-negative curve which is $\cF$-invariant. If $(X,\cF,B,\bM)$ is lc at the generic point of $C$, then $C$ is a rational curve with
\[
    0 < -(K_\cF + B + \bM_X) \cdot C \leq 2.
\]
\end{lemma}

\begin{proof}
    By Corollary~\ref{coro:singinvcurve}, $C$ is not contained in $\Sing^+\cF$. By Remark~\ref{rmk:Clccenter}, $C$ is a log canonical center of $(X,\cF)$, then  $C$ is not contained in the support of $B$ as $(X,\cF,B)$ is lc at the generic point of $C$. In particular, $B \cdot C \geq 0$. By Lemma~\ref{proposition:bdivoncurve}, $\bM_X \cdot C \geq 0 $. Then by \cite[Proposition-Definition 3.12]{CS24}, if $C^\nu \to C$ be the normalization of $C$, then $K_\cF|_{C^\nu} = K_{C^\nu} +\Delta$ for some $\Delta \geq 0$. Therefore
    \begin{align*}
        0 < -(K_\cF + B + \bM_X) \cdot C \leq -K_\cF \cdot C \leq -\deg_{C^\nu} K_{C^\nu} \leq 2.
    \end{align*}
    In particular, $C$ is a rational curve.
\end{proof}

In the end of this section we give the proof of the cone theorem for generalized foliated quadruples on threefolds of rank 1 assuming the ambient variety is $\Q$-factorial.

\begin{theorem}\label{theorem:conethmgfq,Q-fac}
    
    Let $X$ be a normal projective $\Q$-factorial threefold and $(X,\cF,B,\bM)$ be a rank 1 gfq on $X$. Let $Z_{-\infty}$ be the subcone of $\NE(X)$ spanned by the images of $\NE(W) \to \NE(X)$ where $W$ is a non-lc centre of $(X,\cF,B,\bM)$. Let $\{R_j\}_{j \in \Lambda}$ be the set of all $(K_\cF + B + \bM_X)$-negative extremal rays not contained in \( Z_{-\infty}\) in $\NE(X)$. Then the following hold:
    \begin{enumerate}
        \item[(1)] We may write 
        \[
            \NE (X) = \NE(X)_{K_\cF + B + \bM_X\geq 0} + Z_{-\infty} + \sum_{j \in \Lambda} R_j,
        \]
        \item[(2)] For each $j \in \Lambda$, $R_j$ is spanned by a rational curve $C_j$ which is tangent to $\cF$, and 
        \[
            0 < -(K_\cF + B + \bM_X) \cdot C_j \leq 2.
        \]
        \item[(3)] For any ample $\R$-divisor $A$ on $X$, the set 
        \[
        \Lambda_A := \{\,j \in \Lambda \mid R_j \subset \NE(X)_{K_\cF + B +\bM_X + A<0} \,\}
        \]
        is finite. In particular, the $(K_\cF + B + \bM_X)$-negative extremal rays that are not contained in $Z_{-\infty}$ are locally discrete.
    \end{enumerate}
\end{theorem}

\begin{proof}
    \textit{Step 1.} Let $R$ be a $(K_\cF + B + \bM_X)$-negative exposed ray, not contained in \( Z_{-\infty} \). In the first two steps we prove that $R$ is generated by an $\cF$-invariant rational curve $C$ with $0<-(K_\cF + B + \bM_X)\cdot C \leq 2$. Let $H_R$ be a supporting function to $R$. In particular, $H_R$ is nef and not numerically trivial. By \cite[Lemma 8.4.1]{CHLX23}, we may assume that there exists an ample $\R$-divisor $A$ such that $H_R = K_\cF + B + \bM_X + A$ is the supporting function of $R$. 
    
    In this step we deal with the case when $H_R$ is not big. Then there exists an integer $\nu \in \{0,1,2\} $ such that 
\[
H_R^\nu \cdot A^{3-\nu} > H_R^{\nu + 1} \cdot A^{2-\nu} = 0.
\]
Let $D_i = H_R$ for $1 \leq i \leq \nu + 1$, $D_i = A$ for $\nu + 1 < i \leq 3$. Then $D_1 \cdot D_2 \cdot D_3 = H_R^{\nu + 1} \cdot A^{2-\nu } = 0$, and 
\begin{align*}
    -(K_\cF + B + \bM_X) \cdot D_2 \cdot D_3 &= (A - H_R)\cdot H_R^\nu \cdot A^{2-\nu } \\&= H_R^\nu \cdot A^{3-\nu} - H_R^{\nu + 1} \cdot A^{2-\nu} > 0.
\end{align*}
Since $\bM_X$ is pseudo-effective, $\bM_X \cdot D_2 \cdot D_3 \geq 0$. Thus $-(K_\cF + B) \cdot D_2 \cdot D_3 >0 $. Let $N = H_R + A = 2H_R - (K_\cF + B + \bM_X)$, $N$ is an ample $\R$-divisor on $X$.
By Theorem~\ref{thm:B&B} (\cite[Corollary 2.28]{Spi20}), there exists a rational curve $C$ through a general point of $X$, such that $C$ is tangent to $\cF$, and $H_R \cdot C = 0$. In particular, $R$ is generated by $[C]$. 
We may assume $(X,\cF,B,\bM)$ is lc at the generic point of $C$, then by Lemma~\ref{lemma:step1}, $0< -(K_\cF + B + \bM_X) \cdot C \leq 2$.

\medskip

\textit{Step 2.} In this step we deal with the case when $ H_R $ is big. We may write $H_R \sim_\R H + D + tT$ where $H$ is an ample $\R$-divisor, $D \geq 0$ is an $\R$-divisor, $t>0$ and $T$ is a prime divisor not contained in $\Supp D$ such that $T$ is negative along $R$. Let $\nu: S = T^\nu \to T$ be the normalization of $T$.

Assume $T$ is not $\cF$-invariant. Since $T$ is negative along $R$, $R$ is contained in the image of $ \NE(T) \to \NE(X) $. Thus we have that $\mu_T B \leq 1$, since otherwise $T$ would be an non-lc center of $(X,\cF,B,\bM)$ and $R$ would be contained in $Z_{-\infty}$. Let $\alpha = 1-\mu_T B\geq 0$. Then $R$ is also $(K_\cF + B + \alpha T +\bM_X)$-negative. By Proposition~\ref{prop:subadj,primediv}, the restricted foliation $\cF_S$ on $S$ is the foliation by points, and we may write
\[
(K_\cF + B + \alpha T + \bM_X)|_{S}\sim_\R B_S + \bM^S_{S},
\]
where $B_S \geq 0 $ and $\bM^S = \bM|_{T^\nu}$ is the restricted b-divisor of $\bM$ on $ S = T^\nu$. Note that $R$ is contained in the image $\iota: \NE^M(S) \to \NE (X)$. Then by Lemma~\ref{lemma:ExtRayPullback}, there exists a $(B_S + \bM^S_{S})$-negative extremal ray $ R_S \subset \NE^M(S) $ such that $ \iota(R_S) = R$. By Lemma~\ref{lemma:bnefonsurface}, $\bM^S_S$ is nef on $S$, then $B_S$ is negative along $R_S$. Then there exists an irreducible curve $C_S$ contained in $\Supp B_S$ such that $C_S$ is negative along $R_S$. Since $R$ is extremal, $R_S$ is generated by $C_S$, and $R$ is generated by $C = \nu(C_S)$. We may assume $(X,\cF,B,\bM)$ is lc at the generic point of $C$. By Corollary~\ref{coro:singinvcurve}, $C$ is not contained in $\Sing^+ \cF$. Note that $T \cdot C <0$, then by Lemma~\ref{lemma:noninv}, $C$ is $\cF$-invariant. By Lemma~\ref{lemma:step1}, $C$ is rational with $0< -(K_\cF + B + \bM_X) \cdot C \leq 2$.

Now we assume $T$ is $\cF$-invariant. Note that $T$ is not contained in $\Supp B$, since otherwise $T$ would be an non-lc center of $(X,\cF,B,\bM)$ and $R$ would be contained in $Z_{-\infty}$. Thus by Proposition~\ref{prop:subadj,primediv}, we may write
\[
(K_\cF + B + \bM_X )|_{S} \sim_\R K_{\cF_S} + B_S + \bM^{S}_{S},
\]
where $\cF_S$ is the restricted foliation and $B_S \geq 0$. In particular, $(S,\cF_S,B_S,\bM^S)$ is a surface num-gfq. By Lemma~\ref{lemma:ExtRayPullback}, there exists a $(K_{\cF_S} + B_S + \bM_{S}^S)$-negative extremal ray $R_S$ whose image along $\NE^M(S) \to \NE(X)$ is $R$. Apply Proposition~\ref{prop:conethmnumgfq} we conclude that $R_S$ is generated by an irreducible curve $C_S$. Then $R$ is also generated by an irreducible curve $C = \nu(C_S)$. 
We may assume $(X,\cF,B,\bM)$ is lc at the generic point of $C$. By Corollary~\ref{coro:singinvcurve}, $C$ is not contained in $\Sing^+ \cF$. If $C$ is not $\cF$-invariant, by Lemma~\ref{lemma:noninvcurveininvsurf}, we have $\mu_{C_S} B_S \leq 1$. By Proposition~\ref{prop:conethmnumgfq} (1), after replacing $C_S$ we may assume $C_S$ is $\cF_S$-invariant. Hence we may assume that $C$ is $\cF$-invariant. By Lemma~\ref{lemma:step1}, $C$ is rational with $0< -(K_\cF + B + \bM_X) \cdot C \leq 2$.

\medskip

\textit{Step 3.} In this step we prove (3). The proof is similar to the proof of Proposition~\ref{prop:conethmnumgfq}. Assume that for some ample $\R$-divisor $A$ on $X$, $\Lambda_A$ is an infinite set. Let $A_1,\dots,A_{\rho}$ be ample Cartier divisors on $X$ such that $A_1 ,\dots,A_{\rho}$ form a basis for $N^1(X)_\R$. We may assume $A =\epsilon \sum_{k=1}^{\rho} A_k$ for a sufficiently small $\epsilon>0$.

By \cite[Theorem 18.6]{Roc70}, we may shrink $\Lambda_A$ to a countable infinite subset and assume that for each $j \in \Lambda_A$, $R_j $ is exposed. Let $K = K_\cF + B + \bM_X$. By the previous steps, each $R_j$ is generated by a curve $C_j$ with $K\cdot C_j \in [-2,0)$. Since for each $j\in \Lambda_A$, $R_j$ is also $(K + A)$-negative exposed ray, there exists an ample $\R$-divisor $L_j$ such that $H_j = K + A+L_j$ is a supporting function of $R_j$. Then 
\[
0 = H_j \cdot C_j = (K + A+L_j)\cdot C_j \geq K \cdot C_j + A \cdot C_j \geq -2 + \epsilon \sum_{k=1}^{\rho} A_k \cdot C_j,
\]
then $A_k \cdot C_j \leq \frac{2}{\epsilon}$. Since $A_k$ is Cartier, $A_k\cdot C_j\in \Z$. By shrinking $\Lambda_A$ we may assume that $A_k \cdot C_j$ is independent of $j$ for each $k$. Thus for every $j ,j' \in \Lambda_A$, $C_j \equiv C_j'$, which leads to a contradiction. (3) follows. Moreover, following the same proof of Proposition~\ref{prop:conethmnumgfq}, we may see that for every $j \in \Lambda$, $R_j$ is an exposed ray. Then by Step 1 and 2, (1) (2) follows.
\end{proof}

\begin{remark}\label{rmk:WhyCS24notwork}
    It is worth noting that \cite[Theorem 4.8]{CS24} establishes a version of foliated cone theorem for a normal projective $\Q$-factorial variety $X$ and a foliated triple $(X,\cF,B)$ of rank $1$ in arbitrary dimension. One of the key ingredients in the proof of \cite[Theorem 4.8]{CS24} is the adjunction formula along higher-codimensional foliated invariant centers (\cite[Proposition-Definition 3.12]{CS24}). One may naturally expect an analogous version to hold for gfqs of rank 1 in arbitrary dimension, but our techniques fail in the arbitrary-dimensional case. The key difficulty is that it is not clear how to define the analogue of ``restricted b-divisors'' (as in Definition~\ref{def:Restrictedbdiv}) along higher-codimensional foliated invariant centers. Although we only have divisorial adjunctions for generalized foliated structures, which are not sufficient for the arbitrary-dimensional case, they are sufficient enough for the threefold case considered in this paper.
\end{remark}

\begin{corollary}\label{coro: afterconethm,extray}
    Let $X$ be a normal projective $\Q$-factorial threefold and let $(X,\cF,B,\bM)$ be an lc gfq of rank $1$ on $X$. Assume that $K_\cF + B + \bM_X$ is not nef and there exists an ample $\R$-divisor $A$ such that $K_\cF + B + \bM_X + A$ is nef. Let 
    \[
    \lambda := \inf \{\, t > 0 \mid K_\cF + B + \bM_X + tA \text{ is nef}\ \},
    \]
    then there are only finitely many $(K_\cF + B + \bM_X)$-negative extremal rays $R_1,\dots ,R_k$ such that for every $1 \leq i \leq k$, $R_i$ is generated by a curve $C_i$ and $(K_\cF + B + \bM_X + \lambda A) \cdot C_i = 0$.
\end{corollary}

\begin{proof}
    By the assumption, we have $\lambda >0$ and $K_\cF + B + \bM_X + \frac{\lambda}{2} A$ is not nef. By Theorem~\ref{theorem:conethmgfq,Q-fac}, there are only finitely many $(K_\cF + B + \bM_X + \frac{\lambda}{2}A) $-negative extremal rays $R_1, \dots, R_n$. Moreover, for any $ i = 1,2 \dots, n$, there exists a curve $C_i$ such that $R_i = \R_{\geq 0} [C_i]$. By the definition of $\lambda$, we have
    \[
    \lambda = \max_{1\leq i \leq n} \frac{-(K_\cF + B + \bM_X)\cdot C_i}{A\cdot C_i}.
    \]
    In particular, there exist only finitely many $C_j$ such that $(K_\cF + B + \bM_X + \lambda A)\cdot C_j = 0$.
\end{proof}

\section{Minimal Model Program}

In this section we establish the generalized foliated MMP on $\Q$-factorial threefolds and derive the following results as straightforward consequences of \cite{CS20}, Theorem~\ref{theorem:conethmgfq,Q-fac} and Proposition~\ref{proposition:bdivoncurve}. Most of the consequences about MMP for foliated triples of rank 1 on threefolds in \cite{CS20} can be generalized to the case of generalized foliated quadruples.

\begin{proposition}[cf. {\cite[Proposition 8.1]{CS20}}]\label{prop:simpleMMP}
    Let $X$ be a normal threefold, $U$ a normal projective variety, and $\pi: X \to U$ a projective morphism. Let $(X,\cF,B,\bM)$ be a lc gfq of rank $1$ on $X$. Assume that $\cF$ has simple singularities and $K_\cF + B + \bM_X$ is pseudo-effective over $U$. Then we may run a $(K_\cF + B +\bM_X)$-MMP over $U$ to get a minimal model $\phi: X \bir X'$ of $(X,\cF,B,\bM)$ over $U$. If $\cF' = \phi_* \cF$ and $B' = \phi_* B$, then the following holds:
    \begin{enumerate}
        \item[(1)] $\cF'$ has simple singularities.
        \item[(2)] $(X',\cF',B',\bM_{X'})$ is lc. 
        \item[(3)] If $\Theta\geq 0$  is a $\Q$-divisor on $X$ with $\cF$-invariant support such that $(X,\Theta)$ is lc, then $(X',\phi_*\Theta)$ is lc.
    \end{enumerate}
\end{proposition}

\begin{proof}
    First we assmue that $U$ is a point. Since $\cF$ has simple singularities, by Theorem~\ref{thm:simpleprop}, $X$ is $\Q$-factorial klt. We may assume that $K_\cF + B +\bM_X$ is not nef. Then let $R$ be an $(K_\cF + B + \bM_X)$-negative extremal ray. By Theorem~\ref{theorem:conethmgfq,Q-fac}, $R$ is spanned by an $\cF$-invariant rational curve $C$. By Proposition~\ref{proposition:bdivoncurve}, $\bM_X \cdot C \geq 0$. By Remark~\ref{rmk:Clccenter}, $C$ is an lc center of $(X,\cF)$. Since $(X,\cF,B)$ is also lc, $C$ is not contained in $\Supp B$. In particular, $B \cdot C \geq 0 $. Thus $K_\cF \cdot C<0$.

    By \cite[Theorem 6.2]{CS20} and \cite[Theorem 6.5]{CS20}
    , the contraction or the flip associated to $R$ exists. Denote this step of MMP by $\psi: X\bir X_1$ and let $\cF_1 = \psi_*\cF$. Moreover, $\cF_1$ has simple singularities, and $(X_1, \psi_*\Theta)$ is lc. By the negativity lemma, $(X_1,\cF_1,B_1 = \phi_*B, \bM_{X_1})$ is lc. Note that $\bM_{X_1} = \psi_*\bM_X$, $K_{\cF_1} + B_1 + \bM_{X_1} $ is pseudo-effective. Replace $X,\cF,B$ by $X_1,\cF_1,B_1$, we may continue this process. 

    Since $K_\cF + B + \bM_X$ is pseudo-effective, the MMP constructed above will not terminate on a Mori fiber space. Since each step of the MMP is also a $K_\cF$-MMP, then by \cite[Theorem 7.1]{CS20}, the MMP constructed above terminates.

    In the general case, take a sufficiently ample divisor $H$ on $U$. Note that $(X,\cF,B,\bM+\overline{\pi^*H})$ is lc and that $K_\cF + B + \bM_X + \pi^*H$ is pseudo-effective. Thus, we may run a $(K_\cF + B + \bM_X + \pi^*H)$-MMP, which is automatically a $(K_\cF + B + \bM_X)$-MMP over $U$ by Theorem~\ref{theorem:conethmgfq,Q-fac}.
\end{proof}

\begin{lemma}[cf. {\cite[Lemma 8.3]{CS20}}]\label{lemma:non-can,discrep}
    Let $X$ be a normal projective threefold, and let $(X,\cF,B,\bM)$ be an lc gfq of rank $1$ on $X$. Let $E$ be a prime divisor exceptional over $X$ such that $a(E,\cF,B,\bM)<0$. Then $a(E,\cF,B,\bM) = -1$. 
\end{lemma}

\begin{proof}
    Let $f: Y \to X$ be a birational morphism such that $E$ is a divisor on $Y$ and $\bM$ descends to $Y$. Let $\cF_Y = f^{-1}\cF$. By Theorem~\ref{thm:simpleprop}, after replacing by a higher resolution, we may assume that $\cF_Y$ has simple singularities. Let $B_Y = f^{-1}_*B$, then we may write 
    \[
    K_{\cF_Y} + B_Y + \bM_Y + F_1 = f^*(K_\cF + B + \bM_X) + F_2,
    \] 
    where $F_1, F_2 \geq 0$ are $f$-exceptional $\R$-divisors with no common components. After replacing with a higher resolution, we may assume that $(Y,\cF_Y ,B_Y+F) $ is lc (cf. \cite[Corollary, p. 282]{MP13}), where $F$ is the sum of all $f$-exceptional non-$\cF_Y$-invariant prime divisors.

    Now we assume that $a(E,\cF,B,\bM) \in (-1,0)$. Then $\mu_E F_1 \in (0,1)$. Since $(X,\cF,B,\bM)$ is lc, $E$ is not $\cF_Y$-invariant. Since $E$ is exceptional over $X$, $E$ is contained in $\Supp F$, then there exists an $\epsilon>0$ such that $(Y,\cF_Y,B_Y + F_1 +\epsilon E)$ is lc, and so is $(Y,\cF_Y,B_Y + F_1 + \epsilon E,\bM)$ since $\bM$ descends to $Y$. Then by Proposition~\ref{prop:simpleMMP}, we may run a $(K_{\cF_Y}+ B_Y + F_1 + \epsilon E + \bM_Y)$-MMP over $X$ to obtain a minimal model $\phi: Y \bir X'$ over $X$. By the negativity lemma, $\phi$ contracts $E$, which contradicts with \cite[Lemma 6.1]{CS20}.
\end{proof}

\begin{proposition}\label{prop:pltblowup}
    Let $X$ be a normal projective threefold, and let $(X, \cF, B, \bM)$ be an lc gfq of rank $1$ on $X$. Then there exists a birational morphism $\pi: X' \to X$, such that if $\cF' = \pi^{-1} \cF$ and $B' = \pi^{-1}_* B$, then:  
\begin{enumerate}
    \item[(1)] $\cF'$ has simple singularities.  
    \item[(2)] Let $E'$ be the sum of all $\pi$-exceptional prime divisors which are not $\cF'$-invariant, then we may write
    \[
    K_{\cF'} + B' + E' + \bM_{X'} = \pi^*(K_{\cF} + B + \bM_X).
    \]  
    In particular, $(X', \cF',B'+E',\bM)$ is lc.
    \item[(3)] $(X', \Theta)$ is lc, where $\Theta$ is the sum of all $\cF'$-invariant $\pi$-exceptional prime divisors.  
    \item[(4)] If $(X, \cF, B, \bM)$ has an lc center $P$, we may choose $\pi$ such that there exists a $\pi$-exceptional prime divisor $T$ on $X'$ centered on $P$. In particular, $\mu_T E' = \epsilon(T)$.
\end{enumerate}  

\end{proposition}

\begin{proof}
    Let $f: Y \to X$ be a birational morphism such that $\bM$ descends to $Y$. Let $\cF_Y = f^{-1}\cF$. By Theorem~\ref{thm:simpleprop}, after replacing by a higher resolution, we may assume that $\cF_Y$ has simple singularities. Let $f^{-1}_*B = B_Y$, and let $E_1  = \sum_D \epsilon(D)D$, where the sum runs over all $f$-exceptional prime divisors on $Y$. We may write
    \[
    K_{\cF_Y} + B_Y + \bM_Y + E_1 = f^*(K_\cF + B + \bM_X)+E_2,
    \]
    where $ E_2\geq 0 $ is an $f$-exceptional $\R$-divisor. After passing to a higher resolution, we may assume that $(Y,\cF_Y, B_Y + E_1)$ is lc, and $(Y, F)$ is lc where $F$ is the sum of all $\cF_Y$-invariant $f$-exceptional prime divisors (cf. \cite[Corollary, p. 282]{MP13}). 
    Thus $(Y,\cF_Y,B_Y + E_1 ,\bM)$ is lc as $\bM$ descends to $Y$. By Proposition~\ref{prop:simpleMMP}, we may run a $(K_{\cF_Y} + B_Y + E_1 + \bM_{Y})$-MMP over $X$, which terminates at a minimal model $\phi: Y \bir X'$ over $X$. Let $\pi: X' \to X$ be the induced morphism, and define $\cF' = \pi^{-1}\cF$, $B' = \pi^{-1}_* B$, and $E' = \phi_*E_1$. Note that $\phi$ contracts exactly the prime divisors that contained in $\Supp E_2$. Moreover, we have $\Theta = \phi_* F$, we conclude that (1)–(3) hold for $\pi$ by Proposition~\ref{prop:simpleMMP}.

    If $(X,\cF,B,\bM)$ has an lc center $P$, we may assume that there exists a prime divisor $T'$ on $Y$, centered on $P$, such that $a(T',\cF,B,\bM)=-\epsilon(T')$. Then we have $\mu_{T'}E_1 = \epsilon(T') $ and $\mu_{T'}(E_1 - E_2) =-a(T',\cF,B,\bM) = \epsilon(T')$, hence $T'$ is not contained in $\Supp E_2$. It follows that $T' $ is not contracted by $\phi$, and let $T = \phi_* T'$ so (4) follows.
\end{proof}

We give another corollary of Proposition~\ref{prop:simpleMMP} to construct a crepant pullback of any rank 1 gfq on threefolds.

\begin{lemma}\label{lemma:crepantmodel}
Let $X$ be a normal projective threefold, and $(X,\cF,B,\bM)$ be a gfq of rank 1 on $X$. Then there exists a birational morphism $ p : X' \to X$ such that $\cF' = p^{-1}\cF$ has simple singularities, and we may write 
\[
K_{\cF'} + B' + E' + \bM_Y = p^*(K_\cF + B +\bM_X),
\]
where $B' = p^{-1}_*B$, $E'\geq 0$ is an $p$-exceptional $\R$-divisor. 
\end{lemma}

\begin{proof}
    By Theorem~\ref{thm:simpleprop}, let $f: Y \to X$ be a birational morphism such that $\bM$ descends to $Y$ and $\cF_Y = f^{-1}\cF$ has simple singularities. Let $B_Y = f^{-1}_*B$, and we may write 
    \[
    K_{\cF_Y} + B_Y + E + \bM_Y = f^*(K_\cF + B+ \bM_X) + F,
    \]
    where $E,F \geq 0$ are both $f$-exceptional $\R$-divisors and has no common component. Since $\cF_Y$ has simple singularities, then by Theorem~\ref{thm:simpleprop}, $Y$ is $\Q$-factorial klt and $(Y,\cF_Y)$ has canonical singularities. Then $(Y,\cF_Y,0,\bM)$ is also canonical. 
    By Proposition~\ref{prop:simpleMMP}, we may run a $(K_{\cF_Y}+ \bM_Y)$-MMP over $X$ to obtain a minimal model $\phi: Y \bir X'$ over $X$. Let $p: X' \to X$ be the induced morphism, then $p$ is a required morphism. Indeed, let $\cF' = p^{-1}\cF$, $B' = p^{-1}_*B = \phi_* B_Y$, $E' = \phi_* E$, and $F' = \phi_* F$. Note that 
    \[
    K_{\cF'} + \bM_{X'} = p^*(K_\cF + B + \bM_X) + F' - B' -E' \sim_{\R,X} F' - B' -E' 
    \]
    is nef over $X$, and $F'$ is exceptional over $X$, $p_*(F'-B' - E') = -B - p_* E' \leq 0$. By the negativity lemma, $F' - B' -E' \leq 0$, then $F' = 0$ as $ F'$, $B'$ and $E'$ have no common components pairwise. By Proposition~\ref{prop:simpleMMP}, $\cF'$ has simple singularities.
\end{proof}

\begin{proof}[Proof of Theorem~{\upshape\ref{MainThm:conethmgfq}}]

    By Lemma~\ref{lemma:crepantmodel}, there exists a birational morphism $p:X'\to X$ such that we may write
    \[
    K_{\cF'} + B' + \bM_{X'} = p^*(K_\cF + B + \bM_X),
    \]
    where $\cF' = p^{-1}\cF$, $B'\geq 0$ is an $\R$-divisor, and $X'$ is $\Q$-factorial. By Lemma~\ref{lemma:ExtRayPullback}, for each $(K_\cF + B + \bM_X)$-negative extremal ray $R$ which is not contained in $\NE(X)_{\Nlc(X,\cF,B,\bM)}$, there exists a $(K_{\cF'}+B' + \bM_{X'})$-negative extremal ray $R'$ such that $p_*R' = R$ and $R'$ is not contained in $\NE(X')_{\Nlc(X',\cF',B',\bM)}$. Apply Theorem~\ref{theorem:conethmgfq,Q-fac} for $(X',\cF',B',\bM)$, $R'$ is generated by a rational curve $C'$, tangent to $\cF'$, with $0<-(K_{\cF'}+ B' + \bM_{X'} )\cdot C' \leq 2$. Let $C = h(C')$, then $C$ is a rational curve tangent to $\cF$, $R$ is generated by $C$, and 
    \[
    0 < -(K_\cF + B + \bM_X)\cdot C = -(K_{\cF'} + B' + \bM_{X'})\cdot C' \leq 2. \qedhere
    \]
\end{proof}

We also have the divisorial subadjunction for foliated threefolds of rank 1 without $\Q$-factoriality.

\begin{corollary}
    Let $X$ be a normal projective threefold and let $(X, \cF, B ,\bM)$ be a gfq of rank $1$ on $X$. Let $S$ be a prime divisor on $X$ and let $\nu: S^\nu \to S$ be the normalization of $S$. Assume that $S$ is not contained in $\Supp B$. Then there exists a restricted foliation $\cF_S$ with $\rank \cF_S = 1-\epsilon(S)$ and an $\R$-divisor $B_S\geq 0$ such that 
    \[
    (K_\cF + B + \epsilon(S)S + \bM_X)|_{S^\nu} = K_{\cF_S} + B_S + \bM_{S^\nu}^S,
    \]
    where $\bM^S = \bM|_{S^\nu}$ is the restricted b-divisor.
\end{corollary}

\begin{proof}
    By Lemma~\ref{lemma:crepantmodel}, there exists a birational model $p:X' \to X$ such that 
    \[
    K_{\cF'} + B' + \epsilon(T)T + \bM_{X'} = p^*(K_\cF + B +\epsilon(S)S + \bM_X),
    \]
    where $B'\geq 0$ and $T = p^{-1}_* S$ is not contained in $\Supp B'$. Moreover, $\cF' = p^{-1}\cF$ has simple singularities, and $X'$ is $\Q$-factorial. By Proposition~\ref{prop:subadj,primediv}, there exists a restricted foliation $\cF_T$ of rank $1-\epsilon(S)$ on $T^\nu$ and an $\R$-divisor $B_T \geq 0$ on $T^\nu$ such that
    \[
    \left(K_{\cF'} + B' + \epsilon(T)T + \bM_{X'}\right)|_{T^\nu} = K_{\cF_{T}} + B_{T} + \bM^{T}_{T_\nu},
    \]
    where $\bM^T = \bM|_{T^\nu}$ is the restricted b-divisor. Let $q:T^\nu \to S^\nu $ be the induced morphism. It is clear that $\bM^T_{S^\nu} = \bM^S_{S^\nu}$, where $\bM^S = \bM|_{S^\nu}$. Let $\cF_S = q_*\cF_T$ and $B_S =q_* B_T $, then $B_S \geq 0$ and we may conclude.
\end{proof}

The following proof of Theorem~\ref{MainThm:FoliatedMMP} is a direct corollary of \cite[Theorem 8.10]{CS20}
, similarly to the proof of Proposition~\ref{prop:simpleMMP}. 

\begin{proof}[Proof of Theorem~{\upshape\ref{MainThm:FoliatedMMP}}]
    If $K_\cF + B + \bM_X$ is pseudo-effective, we may assume that $K_\cF + B +\bM_X$ is not nef. Then let $R$ be a $(K_\cF + B + \bM_X)$-negative extremal ray. By Theorem~\ref{theorem:conethmgfq,Q-fac}, $R$ is spanned by an $\cF$-invariant rational curve $C$. Similarly to the proof of Proposition~\ref{prop:simpleMMP}, $K_\cF \cdot C<0$, and $R$ is also an $K_\cF$-negative extremal ray. We only need to show that the contraction or the flip associated to $R$ exists, which has been established in \cite[Theorem 8.10]{CS20}.

    If $K_\cF + B + \bM_X$ is not pseudo-effective, then $K_\cF$ is not pseudo-effective. By \cite[Theorem 1.1]{CP19} or \cite[Lemma 2.21]{CS21}, $\cF$ is algebraically integrable. Then we may conclude by \cite[Theorem A.6]{liu_minimal_2024}.
\end{proof}

\section{Base-point-free theorem}

In this section, we prove Theorem~\ref{MainThm:BPF}, a version of the base-point-free theorem for foliated threefolds of rank one. In fact, using the similar idea of \cite[Theorem 1.1]{CS25bpf} and the generalized foliated structure, we prove the following version of the base-point-free theorem for generalized foliated quadruples:

\begin{theorem}\label{Thm:BPFgfq}
    Let $X$ be a normal projective threefold, and let $(X,\cF ,B,\bM)$ be an lc gfq of rank $1$ on $X$. Assume that $X$ is $\Q$-factorial klt. Let $A$ be an ample $\R$-divisor on $X$. Suppose that $H = K_\cF + B + \bM_X + A$ is nef. Then $H$ is semi-ample.
\end{theorem}

\begin{proof}
    If $H$ is not big, then $K_\cF$ is not pseudo-effective. By \cite[Theorem 1.1]{CP19} or \cite[Lemma 2.21]{CS21}, $\cF$ is algebraically integrable. 
    By \cite[Theorem A.8]{liu_minimal_2024}, $H$ is semi-ample.

    From now we assume that $H$ is big. If $K_\cF + B + \bM_X + \alpha A$ is nef for some $\alpha \in (0,1)$, then $H = (K_\cF+ B+ \bM_X +\alpha A) + (1-\alpha) A$ ample and we may conclude. Thus we may assume that
    \[
    \lambda = \inf \{\, t > 0 \mid K_\cF + B + \bM_X + tA \text{ is nef}\, \} = 1.
    \]
    By Corollary~\ref{coro: afterconethm,extray}, there exist only finitely many $(K_\cF + B + \bM_X)$-negative extremal rays $R_1, \dots, R_k$ such that $H$ vanishes along $R_i$ for each $1 \leq i\leq k$. By Theorem~\ref{theorem:conethmgfq,Q-fac}, $R_i$ is generated by an $\cF$-invariant curve $C_i$. Since $H$ is big, $\loc R_i\neq X$ for each $1 \leq i \leq k$. By Corollary~\ref{coro:singinvcurve}, $C_i$ is not contained in $\Sing^+\cF$. By the proof of Lemma~\ref{lemma:step1}, for each $i$ we have $K_\cF \cdot C_i <0$.

    By \cite[Corollary 8.5]{CS20}
    , we may assume that $\cF$ is canonical along $C_i$ for each $i$. If $\loc R_i$ has dimension $2$ for some $i$, then by \cite[Theorem 6.2]{CS20}
    , the divisorial contraction $\phi: X\to Y$ associated to $R_i$ of $X$ exists. Take a sufficiently small ample $\R$-divisor $D$ on $Y$ such that $A - \phi^* D$ is ample. Let $\Theta\geq 0$ be a $\R$-divisor such that $A - \phi^* D \sim_\R \Theta$. Since $(X,\cF,B,\bM)$ is lc, then so is $(X,\cF,B,\bM + \overline{\Theta + \phi^*D})$. Note that as b-divisors, $\overline{\phi^*D} = \overline{D}$, and since $H$ is numerically trivial over $Y$, $K_\cF + B + \bM_X + \Theta + \phi^*D$ is also numerically trivial over $Y$. Then we have
    \[
    K_\cF + B + \bM_X + \Theta + \phi^*D = \phi^*(K_{\cF_Y} + B_Y +\bM_Y +  \Theta_Y + D),
    \]
    where $\cF_Y = \phi_*\cF$, $B_Y = \phi_* B$ and $\Theta_Y = \phi_*\Theta$. In particular, $(Y,\cF_Y,B_Y,\bM + \overline{\Theta} + \overline{D})$ is lc. Let $\bM^Y = \bM + \overline{\Theta}$, then $\bM^Y$ is b-nef and $(Y,\cF_Y,B_Y,\bM^Y)$ is lc. Then we may replace $(X,\cF,B,\bM),A$ with $(Y,\cF_Y,B_Y,\bM^Y),D$ and apply induction on the Picard number of $X$.

    Now we assume that when $Z_i := \loc R_i$ has dimension $1$ for every $1 \leq i \leq k$. We claim that for any surface  $T \subset X$, $H|_T$ is big. Suppose not for some surface $T \subset X$. If $T$ is not $\cF$-invariant, take $0< \epsilon \ll 1 $ such that $H - \epsilon A $ is big. Let $\Gamma \sim_\R H-\epsilon A$ be an effective $\R$-divisor. Since $H|_T$ is not big, $\Gamma|_T$ is not pseudo-effective. In particular, $T \subset \Supp \Gamma$ and there exists a general curve $C$ on $T$ which is not $\cF$-invariant such that $\Gamma \cdot C <0$. We may take $C$ sufficiently general such that $C$ is not contained in any components of $\Supp \Gamma$ other than $T$. By Lemma~\ref{lemma:noninv}, $(K_\cF + B +\bM_X)\cdot C \geq 0$. However,
    \[
    (K_\cF + B + \bM_X)\cdot C = (H - A)\cdot C = (\Gamma - (1-\epsilon)A)\cdot C <0,
    \]
    which leads to a contradiction. Hence $T$ is $\cF$-invariant. Let $S \to T$ be the normalization map. By Proposition~\ref{prop:subadj,primediv}, there exists $\Delta_S \geq 0$ on $S$ such that 
    \[
    H|_S = (K_\cF + B + \bM_X + A)|_S = K_{\cF_S} + \Delta_S + \bM^S_S,
    \]
    where $\cF_S$ is the restricted foliation. Since $H|_S$ is not big, by Theorem~\ref{thm:B&B} there exists a general curve $\xi_S$ on $S$ such that $H|_S \cdot \xi_S = 0$. Let $\xi$ be the image of $\xi_S$ on $X$, then $H\cdot \xi = 0$. Since $\xi$ is general on $T$ and $H$ is big, $T \cdot \xi<0$. In particular, there exists $1\leq j \leq k$ such that $T \cdot R_j <0$. In particular, we have $\loc R_j \subset T$. Since $\loc R_j$ is of dimension $1$, $[\xi]$ is not contained in $ R_j$. By \cite[Lemma 3.2]{CS25bpf}, 
    through any general point $x$ of $X$ there exists a curve $\xi_x$ such that $H \cdot \xi_x = 0$, which contradicts to the assmuption that $H$ is big, and the claim follows.

    By \cite[Lemma 3.2]{CS25bpf}, we may assume that $Z_i, Z_j$ are disjoint for each $i \neq j$. Let $\Sigma = \bigcup_{i=1}^k Z_i$. By \cite[Lemma 5.4]{CS20} and \cite[Proposition 2.29]{CS20}, the normal bundle of $\Sigma$ is anti-ample. Then by \cite[Theorem 6.2]{Art70}, we may construct a morphism $f: X \to Z$ of algebraic spaces which contracts each connected component of $Z_i$ to a point. Since $R_i$ is $H$-trivial for each $1 \leq i \leq k$, by the construction of flips in \cite[Theorem 8.8]{CS20}, there exists an $\R$-Cartier divisor $H_Z$ on $Z$ such that $H = f^* H_Z$. In particular, $H_Z$ is nef on $Z$. 
    
    If $H_Z$ is ample on $Z$, then $Z$ is a projective variety and $H$ is semi-ample, hence we may conclude. Then we may assume that $H_Z$ is nef but not ample. Note that $H_Z|_{T'}$ is big for any surface $T'$ on $Z$. Therefore by the Nakai-Moishezon criterion, there exists an extremal ray $R_{0,Z}$ of $\NE(Z)$ such that $H_Z$ vanishes along $R_{0,Z}$. Then by Lemma~\ref{lemma:ExtRayPullback}, there exists an $H$-trivial extremal ray $R_0$ on $X$ such that the $\iota(R_0) = R_{0,Z} $ where $\iota: \NE(X) \to \NE(Z)$. Since $\iota(R_i) = 0$ for every $1 \leq i \leq k$, $R_0$ is not any of $R_i$, which contradicts to the choice of those $R_i$.
\end{proof}

\begin{proof}[Proof of Theorem~{\upshape\ref{MainThm:BPF}}]
    It follows from Theorem~\ref{Thm:BPFgfq} for $\bM = \overline{0}$.
\end{proof}

\section{ACC for log canonical thresholds}\label{section:ACCforLCT}

\subsection{Precise Adjunction Formula} We give a precise subadjunction formula to an invariant surface for generalized foliated quadruples of rank 1 on threefolds, provided that the foliation has simple singularities.

\begin{proposition}[cf. {\cite[Theorem 3.3]{LMX24}}]\label{prop:Pre.AdjF.Simple}
    Let $X$ be a normal threefold and $(X,\cF,B,\bM) $ be a gfq of rank $1$ on $X$ such that $\cF$ has simple singularities. Let $T,B_1,\dots,B_m$ be distinct prime divisors on $X$, and let $\bM^1 ,\dots,\bM^n$ be b-nef b-Cartier b-divisors on $X$. Suppose that $\epsilon (T) = 0$ and 
    \[
    B = \sum_{i=1}^m b_iB_i, \quad \bM = \sum_{j=1}^n r_j\bM^j,
    \]
    where $b_1,\dots,b_m,r_1,\dots r_n$ are nonnegative real numbers. Let $S\to T$ be the normalization of $T$ and let $\bM^{j,S} = \bM^j |_{S}$ be the restricted b-divisors. Then there exist a foliation $\cF_S$ of rank $1$ on $S$, prime divisors $C_1,\dots, C_q$ on $S$, and nonnegative integers $\{w_{i,j}\}_{1\leq i \leq q,0\leq j\leq m}$ and $\{v_{i,j}\}_{1\leq i \leq q, 1\leq j \leq n}$ such that for any real numbers $b_1',\dots, b_m'$, $r_1',\dots r_n'$, we have the following holds: Let 
    \[
    B' = \sum_{i=1}^m b_i'B_i, \quad \bM' = \sum_{j=1}^n r_j'\bM^j,
    \]
    we may write 
    \[
    (K_{\cF} + B' + \bM'_X)|_{S} = K_{\cF_S} + B'_S + \bM'^S_{S},
    \]
    where 
    \[
    B'_S = \sum_{i=1}^q \frac{w_{i,0} + \sum_{j=1}^m w_{i,j}b_j' + \sum_{j=1}^n v_{i,j}r_j'}{2} C_i, \quad \bM'^S = \sum_{j=1}^n r_j' \bM^{j,S}.
    \]
    In particular, if there exists an $I\subset [0,+\infty)$ such that $1, b_1',\dots, b_m'$, $r_1',\dots r_n' \in I$, then the coefficients of $B_S$ are contained in $\frac{1}{2}I_+$. 
\end{proposition}

\begin{proof}
    Since $\cF$ has simple singularities, $X$ is $\Q$-factorial. For any prime divisor $D$ on $S$, let $\eta_D$ be the generic point of the image of $D$ in $X$.  By \cite[Theorem 3.3]{LMX24}, we only need to show that for any prime divisor $D$ on $S$, we have non-negative integers $v_1,\dots v_n$ such that near $\eta_D$ we have
    \[
    \left(\sum_{j=1}^n r_j'\bM^j_{X} \right) \Bigg|_{S} =  \frac{\sum_{j=1}^n v_jr_j'}{2}  D + \sum_{j=1}^n r_j' \bM^{j,S}_{S}.
    \]
    Here we point out that the proof of \cite[Theorem 3.3]{LMX24} does not require the lc assumption. By \cite[Claim 3.4]{LMX24}, for any prime divisor $L$ on $X$, $2L$ is Cartier near the generic point of the image of $D$ in $X$. Let $f:Y \to X$ be a log resolution of $(X,T)$ such that $\bM^j$ descends to $Y$ for each $1\leq j\leq n$. 
    Let $S' = f^{-1}_*T$, let $g: S' \to S$ be the induced morphism, and let $D' = g^{-1}_*D$.
    By the negativity lemma, we may write 
    $$f^*(2\bM^{j}_X) = 2\bM^{j}_Y + F_j, $$
    where $F_j\geq 0$. Then $F_j$ is Cartier near the generic point of the image of $D'$ on $Y$, since $f^*(2\bM^j_X)$ is Cartier near the generic point of the image of $D'$ on $Y$ and $\bM^j_Y$ is Cartier on $Y$. In particular, $v_j = \mu_{D'}(F_j|_{S'}) $ is a non-negative integer. Then we have
    \[
    (\bM^{j}_X)|_{S} =  \frac{v_j}{2}D + \bM^{j,S}_{S}
    \]
    near $\eta_D$. Therefore, the non-negative integers $v_1,\dots,v_n$ are as required.
\end{proof}

The following is a technical lemma needed in the proof of Theorem~\ref{MainThm:ACCforLCT}. It is inspired by the proof of \cite[Theorem 4.6]{Che22} or \cite[Claim 4.7]{LMX24}. It can be regarded as a computation of the coefficients of the different for adjunction to some invariant curves.

\begin{lemma}\label{lemma:forACC.adj.curves}
    Let $X$ be a normal threefold, and let $\cF$ is a foliation of rank $1$ on $X$ with simple singularities. Suppose $T\subset X$ is an $\cF$-invariant prime divisor  and let $\iota: S \to T$ be the normalization of $T$. Let $\cF_S$ be the restricted foliation of $\cF$ on $S$ and assume that $\cF_S$ is algebraically integrable. let $L$ be a general leaf of $\cF_S$, and let $L_X$ be the image of $L$ on $X$. Then the following hold:
    \begin{enumerate}
        \item[(1)] Let $P\in L$ be any closed point. If $\iota(P)$ is an lc center of $\cF$, then for any prime divisor $D$ on $X$, $2D$ is Cartier near $\iota(P)$. If $\iota(P)$ is not an lc center of $\cF$, then $\cF$ is terminal at $\iota(P)$, and $L$, as a divisor on $S$, is Cartier near $P$.
        \item[(2)] There exist non-negative integers $\lambda, q$ and positive integers $p_1, \dots, p_q$ such that
        \[
            K_{\cF}\cdot L_X = \deg K_{L^\nu} + \frac{\lambda}{2} + \sum_{k=1}^q \frac{p_k-1}{p_k},
        \]
        where $L^\nu \to L$ is the normalization of $L$.
        \item[(3)] For any prime divisor $B\neq T$ on $X$, and for any b-Cartier b-nef b-divisor $\bM$ on $X$, both $2B \cdot L_X$ and $2\bM_X \cdot L_X$ are non-negative integers.
    \end{enumerate}
\end{lemma}

\begin{proof}
    Since $\cF$ has simple singularities, $X$ is $\Q$-factorial. Since $\cF_S$ is algebraically integrable, by \cite[3.5]{Dru21} we may assume that there exists a birational morphism $f: V\to S$ such that if $\cF_V= f^{-1}\cF_S$, there exists a fibration $g: V\to Z$ to a curve $Z$ such that $\cF_V$ is induced by $g$. 

    Let $P\in L$ be any closed point. If $\iota(P)$ is an lc center of $\cF$, then either $X$ is smooth at $\iota(P)$ or $\iota(P)$ is a $\Z/2\Z$-quotient singularity of $X$. In particular, for any prime divisor $D$ on $X$, $2D$ is Cartier near $\iota(P)$. If $\iota(P)$ is not an lc center of $\cF$, then by the definition of simple singularities, $\cF$ is terminal at $\iota(P)$. By \cite[Lemma 2.9]{CS20}
    , $\cF$ is induced by a fibration up to a quasi-\'etale cover near $\iota(P)$. Then there are only finitely many $\cF$-invariant curves on $X$ passing through $\iota(P)$, hence there are also only finitely many $\cF_S$-invariant curves on $S$ passing through $P$. In particular, $\cF_S$ is non-dicritical at $P$. Since $L$ is a general leaf of $\cF_S$, $f$ is an isomorphism over a neighborhood of $P$ and $L_V :=f^{-1}_*L$, as a general fiber of $g$, is a Cartier divisor on $V$. In particular, $L$ is Cartier near $P$. (1) follows. Moreover, (2) follows from (1) and \cite[Proposition 2.13]{CS20}.

    Now we prove (3). Since $B\neq T$ and $L$ is a general leaf of $\cF_S$, $B\cdot L_X\geq 0$. We may write 
    \[
        2\bM_X|_S =  2\bM^S_S  + \Theta,
    \]
    where $\bM^S = \bM|_S$ is the restricted b-divisor. By Proposition~\ref{prop:Pre.AdjF.Simple}, $\Theta$ is an effective Weil divisor on $S$. By Lemma~\ref{lemma:bnefonsurface}, $\bM^S_S \cdot L\geq 0$. Since $L$ is a general leaf, $\Theta \cdot L \geq 0$. Then 
    \[
    2\bM_X \cdot L_X = (2\bM_X)|_S \cdot L = 2\bM^S_S \cdot L + \Theta \cdot L \geq 0. 
    \] 
    It suffices to show that $2B \cdot L_X, 2\bM_X \cdot L_X \in \Z$. Note that $2B \cdot L_X = (2B)|_S \cdot L$ and $2\bM_X \cdot L_X = (2\bM_X)|_S \cdot L$ by the projection formula. By Proposition~\ref{prop:Pre.AdjF.Simple}, $2B|_S$ is a Weil divisor on $S$. Since $\bM$ is b-Cartier, $\bM^S$ is also b-Cartier. Then $\bM^S_S$ is a Weil divisor, hence so is $2\bM_X|_S$. Therefore by (1), for any closed point $P\in L$, either $2B$ (resp. $2\bM_X$) is Cartier near $\iota(P)$, in which case $2B|_S$ (resp. $2\bM_X|_S$) is Cartier near $P$, or $L$ is Cartier near $P$. It follows that $(2B)|_S \cdot L$ (resp. $(2\bM_X)|_S \cdot L$) is an integer. (3) follows.
\end{proof}

\subsection{Proof of Theorem~{\ref{MainThm:ACCforLCT}}}

\begin{lemma}\label{lemma:ACC1}
    Let $X$ be a normal projective threefold and let $(X, \cF,B,\bM)$ be an lc gfq of rank $1$ on $X$. Let $D$ be an effective $\R$-divisor on $X$ and $\bN$ be a b-nef b-divisor on $X$ such that $D + \bN_X$ is $\R$-Cartier. Assume that 
    \begin{itemize}
        \item $(X,\cF,B+D,\bM+\bN)$ is lc,
        \item $ (X,\cF,B + tD, \bM + t\bN)$ is not lc for any $t >1$, and
        \item $\lfloor B\rfloor = \lfloor B+ D\rfloor$.
    \end{itemize}
    Then there are birational morphisms $g: Y\to X'$ and $\pi:X' \to X$ such that 
    \begin{enumerate}
        \item[(1)] Let $\cF' = \pi^{-1}\cF$, let $B'$ and $ D'$ be the strict transforms of $B$ and $D$ on $X'$, respectively, and let $E' $ be the sum of all $\pi$-exceptional prime divisors which are not $\cF'$-invariant. Then we may write
        \[
        K_{\cF'} + B' + D' + E' + \bM_{X'} + \bN_{X'} = \pi^*(K_\cF + B + D + \bM_X  + \bN_X),
        \]
        and 
        \[
            K_{\cF'} + B'+ E' + \bM_{X'}= \pi^*(K_\cF + B + \bM_X ).
        \]
        \item[(2)] The birational morphism $g:Y \to X'$ is a divisorial contraction of a prime divisor $T$. Let $\cF_Y= g^{-1}\cF'$, then $T$ is $\cF_Y$-invariant. Moreover, $a(T,\cF,B+D,\bM+\bN)=0$ and $a(T,\cF,B,\bM)>0$. In particular, $-T$ is ample over $X'$.
        \item[(3)] Both $\cF'$ and $\cF_Y$ have simple singularities. In particular, $X'$ and $Y$ are $\Q$-factorial.
    \end{enumerate}
\end{lemma}

\begin{proof}
    Since $(X,\cF,B+D ,\bM + \bN)$ is lc, both $B$ and $D$ have no $\cF$-invariant component. Then by the assumptions, there exists an lc center $Q$ of $(X,\cF,B+D,\bM + \bN)$ of codimension $\geq 2$ which is not an lc center for $(X,\cF,B,\bM)$. 

    By Proposition~\ref{prop:pltblowup}, there exists a birational morphism $\pi':X'' \to X$ such that we may write
    \[
    K_{\cF''} + B'' + D'' + E'' + \bM_{X''} + \bN_{X''} = \pi'^*(K_\cF + B + D + \bM_X + \bN_X),
    \]
    where $\cF'' = \pi'^{-1}\cF$, $B''$, $D''$ are the strict transforms of $B$, $D$ respectively on $X''$, and $E''$ is the sum of all $\pi'$-exceptional prime divisors which are not $\cF''$-invariant. We have that $\cF''$ has simple singularities and $X''$ is $\Q$-factorial. 
    
    For any $\Gamma \subset \Supp E''$,  we have 
    \[
    a(\Gamma,\cF,B + D, \bM + \bN) = -1.
    \]
    Assume that $a(\Gamma,\cF,B,\bM)>-1$. Since $(X,\cF,B,\bM)$ is lc, by Lemma~\ref{lemma:non-can,discrep}, we have 
    \[
    a(\Gamma):= a(\Gamma,\cF,B,\bM) \geq 0.
    \]
    Then for $0 < \delta < \frac{1}{a(\Gamma) + 1}$, we have
    \begin{align*}
        & a(\Gamma, \cF, B + (1-\delta)D,\bM +(1-\delta)\bN)\\ = \,& \delta a(\Gamma, \cF, B,\bM) + (1-\delta) a(\Gamma, \cF, B + D,\bM +\bN) \\ = \,& \delta (a(\Gamma)+1) - 1 \in (-1,0),
    \end{align*}  
    which contradicts with Lemma~\ref{lemma:non-can,discrep} and the fact that $(X,\cF, B + (1-\delta)D,\bM +(1-\delta)\bN)$ is lc. Thus $a(\Gamma,\cF,B,\bM) = -1$.

    By the negativity lemma, we may write 
    \[
    D'' + P'' +  \bN_{X''}  = \pi'^*(D + \bN_{X}),
    \]
    where $P''\geq 0$ is a $\pi'$-exceptional $\R$-divisor on $X''$. Then we have 
    \[
    K_{\cF''} + B''+ E'' -P''+\bM_{X''} = \pi'^*(K_\cF + B + \bM_X).
    \]
    Note that $P''$ and $E''$ have no common component as $a(\Gamma,\cF,B ,\bM) = -1$ for each prime divisor $\Gamma \subset \Supp E''$. In particular, each component of $P''$ is $\cF''$-invariant and exceptional over $X$. Moreover, by Proposition~\ref{prop:pltblowup} (4), we may assume that there exists a prime divisor $T''$ on $X''$, centered on $Q$, such that $\mu_{T''}E'' = \epsilon(T'')$. Then $T''$ must be $\cF''$-invariant and contained in $\Supp P''$, as $Q$ is not an lc center for $(X, \cF, B, \bM)$. In particular, $P'' \neq 0$. 

    Since $(X'',\cF'',B''+D'' +E'',\bM +  \bN)$ is lc and $X''$ is $\Q$-factorial, we have $(X'',\cF'',B'' +E'',\bM )$ is lc. By Proposition~\ref{prop:simpleMMP}, we may run a $(K_{\cF''}+ B'' + E'' + \bM_{X''})$-MMP over $X$ which terminates at a minimal model $\phi:X'' \bir X'$ over $X$. Let $\pi:X' \to X$ be the induced morphism. Note that $\phi$ contracts exactly prime divisors contained in $\Supp P''$. Let $\cF' = \pi^{-1}\cF$, and let $B'$, $D'$, $E'$ be the strict transform of $B''$, $D''$, $E''$ on $X'$, respectively. We have
    \[
    K_{\cF'} + B' + E' + \bM_{X'}  = \pi^*(K_\cF + B + \bM_X).
    \]
    Note that 
    \[
    K_{\cF''} + B'' + D'' + E'' + \bM_{X''} + \bN_{X''}\sim_{\R,X} 0,
    \]
    hence (1) holds. Moreover, since $P''\neq 0$, some step of this MMP must be a divisorial contraction. Since $K_{\cF'} + B' + E' + \bM_{X'} \sim_{\R,X} 0$, the last step of this sequence of MMP is a divisorial contraction. Let $g: Y \to X'$ be the divisorial contraction which is the last step of the sequence of MMP associated with the $g$-exceptional prime divisor $T$. Note that the strict transform of $T$ on $X''$ is contained in $\Supp P''$, hence $\epsilon(T) = 0$, $a(T,\cF,B+D,\bM+\bN)=0$, and $a(T,\cF,B,\bM)>0$. In particular, $-T$ is ample over $X'$ and (2) follows. (3) follows by Proposition~\ref{prop:simpleMMP}. 
\end{proof}

\begin{proof}[Proof of Theorem~{\upshape\ref{MainThm:ACCforLCT}}]
    \textit{Step 1}. Assume the theorem does not hold, then there exists a sequence of rank $1$ gfqs $(X_i,\cF_i,B_i,\bM^i)$ and $D_i$, $\bN^i$ satisfying the assumptions of the theorem such that if 
    \[
    t_i = \lct (X_i,\cF_i,B_i,\bM^i;D_i,\bN^i),
    \]
    then $\{t_i\}_{i=1}^\infty$ form a strictly increasing sequence of numbers. We may assume $0<t_i<+\infty$ for each $i$. Then $(X_i,\cF_i,B_i+ t_i D_i,\bM^i + t_i \bN^i)$ is lc, and $(X_i,\cF_i,B_i+ t D_i,\bM^i + t \bN^i)$ is not lc for $t >t_i$. In particular, both $B_i$ and $D_i$ have no $\cF_i$-invariant components. If $\lfloor B_i + t_iD_i \rfloor \neq \lfloor B_i \rfloor $ for infinitely many $i$, then after replacing $\{t_i\}$ with a subsequence, we may assume that for each $i$, $\lfloor B_i + t_iD_i \rfloor \neq \lfloor B_i \rfloor $, and there exists a prime divisor $\Gamma_i$ such that $\mu_{\Gamma_i}(B_i + t_iD_i) = 1 $ and $\mu_{\Gamma_i}B_i <1$. Then we have \[
    t_i = \frac{1-\mu_{\Gamma_i}B_i}{\mu_{\Gamma_i}D_i}.
    \]
    Since $\mu_{\Gamma_i}B_i$, $\mu_{\Gamma_i}D_i\in I$, we have that $t_i$ belongs to a fixed ACC set which leads to a contradiction. Thus we may assume that $\lfloor B_i + t_iD_i \rfloor = \lfloor B_i \rfloor $ for each $i$. In particular, there exists an lc center $Q_i$ of $(X_i,\cF_i, B_i + t_iD_i,\bM^i + t_i\bN^i)$ of codimension $\geq 2$ which is not an lc center for $(X_i,\cF_i,B_i,\bM^i) $. 

    \medskip

    \textit{Step 2.} 
    In this step we replace $X_i$ with a suitable birational model. By Lemma~\ref{lemma:ACC1}, for each $i$, there exists birational morphisms $\pi_i: X_i' \to X_i$ and $g_i:Y_i \to X_i'$ such that the following hold:
    \begin{itemize}
        \item Let $\cF_i' = \pi_i^{-1}\cF_i$, let $B_i'$ and $ D_i'$ be the strict transforms of $B_i$ and $D_i$,  respectively, and let $E_i' $ be the sum of all $\pi_i$-exceptional prime divisors which are not $\cF_i'$-invariant. Then $\cF_i'$ has simple singularities, and we may write
        \[
            K_{\cF_i'} + B_i' + E_i' + t_iD_i' + \bM^i_{X_i'} + t_i\bN^i_{X_i'} = \pi^*(K_{\cF_i} + B_i + t_iD_i + \bM^i_{X_i}  + t_i\bN^i_{X_i}),
        \]
        and
        \[
            K_{\cF_i'} + B_i' + E_i' + \bM^i_{X_i'} = \pi^*(K_{\cF_i} + B_i + \bM^i_{X_i}).
        \]
        \item The birational morphism $g:Y_i \to X_i'$ is a divisorial contraction of a prime divisor $T_i$,  $\cF_{Y_i}= g_i^{-1} \cF_i'$ has simple singularities, and $T_i$ is $\cF_{Y_i}$-invariant. Moreover, $a(T_i,\cF_i,B_i+t_iD_i,\bM^i+t_i\bN^i)=0$ and $e_i:=a(T_i,\cF_i,B_i,\bM^i)>0$. In particular, $-T_i$ is ample over $X_i'$.
    \end{itemize}
    In particular, $t_i = \lct(X_i',\cF_i', B_i' + E_i', \bM^i; D_i', \bN^i)$. We may assume that $1 \in I$. After replacing $(X_i,\cF_i,B_i,\bM^i;D_i,\bN^i)$ with $(X_i',\cF_i', B_i' + E_i', \bM^i; D_i', \bN^i)$, we may assume that $\cF_i$ has simple singularities and there exists a birational morphism $g_i: Y_i \to X_i$ which is a divisorial contraction of the prime divisor $T_i$ with $\epsilon(T_i) = 0$. Let $\cF_{Y_i} = g^{-1}\cF_i$, and let $B_{Y_i}$, $D_{Y_i}$ be the strict transform of $B_i$ and $D_i$ on $Y_i$ respectively, we may write
    \[ 
    K_{\cF_{Y_i}} + B_{Y_i} + t_i D_{Y_i} + \bM^i_{Y_i} + t_i \bN^i_{Y_i} = g_i^* ( K_{\cF_i} + B_i +t_iD_i + \bM^i_{X_i} + t_i \bN^{i}_{X_i}),    
    \]
    and 
    \[
        K_{\cF_{Y_i}} + B_{Y_i} -e_i T_i + \bM^i_{Y_i}  = g_i^* ( K_{\cF_i} + B_i  + \bM^i_{X_i} ).
    \]
    In particular, $-T_i$ is ample over $X_i$. Moreover, by the two equalities above, we have
    \[
    D_{Y_i} + \bN^i_{Y_i} +\frac{e_i}{t_i}T_i = g_i^*(D_i + \bN^i_{X_i}).
    \]

    \medskip

    \textit{Step 3.} In this step we apply the divisorial adjunction to $T_i$ and introduce some necessary notations. Let $S_i \to T_i$ be the normalization of $T_i$. Apply Proposition~\ref{prop:Pre.AdjF.Simple} for $(Y_i,\cF_{Y_i},B_{Y_i} + t_iD_{Y_i},\bM^i + t_i \bN^i)$ and $T_i$, we may write

    \[
    K_{\cF_{S_i}} + \Delta_i + \mathbf{\Psi}^i_{S_i} := (K_{\cF_{Y_i}}+ B_{Y_i}+t_iD_{Y_i}+ \bM^i_{Y_i} + t_i \bN^i_{Y_i})|_{S_i},
    \] 
    where $\cF_{S_i}$ is the restricted foliation of $\cF_{Y_i}$ on $S_i$, $\Delta_i \geq 0$, and $\mathbf{\Psi}^i = \bM^{S_i} + t_i \bN^{S_i}$. Here $\bM^{S_i} = \bM^i|_{S_i}$, $\bN^{S_i} = \bN^i|_{S_i}$. In particular, $\bM^{S_i}, \bN^{S_i}$ are both $I$-linear combination of b-nef b-Cartier b-divisors on $S_i$. 

    Now we give some notations used in next steps. We define
    \[
    K_i := \{\,a + t_ib \mid a,b\in I\cup \{0\}\, \},
    \]
    then $K_i$ is a DCC set. The coefficients of $B_{Y_i}+ t_iD_{Y_i}$ belong to $K_i$, and $\bM^i + t_i\bN^i$ is a $K_i$-linear combination of b-nef b-Cartier b-divisors. Then by Proposition~\ref{prop:Pre.AdjF.Simple}, the coefficients of $\Delta_i$ belong to $\frac{1}{2}(K_i)_+$. Let $K = \bigcup_{i=1}^\infty K_i$, then $K$ is a DCC set as $\{\,t_i\mid i \geq 1\,\}$ is DCC. In particular, $\Delta_i \in \frac{1}{2}K_+$.

    By Proposition~\ref{prop:Pre.AdjF.Simple}, there exists a $\Q$-divisor $G_i\geq 0$ on $S_i$ such that $K_{\cF_{S_i}} + G_i = K_{\cF_{Y_i}}|_{S_i}$. Let $D_{S_i} = D_{Y_i}|_{S_i}$. By the negativity lemma, we may write $\bM^i_{Y_i}|_{S_i} = \bM^{S_i}_{S_i} + \Theta_i$ and $\bN_{Y_i}^i |_{S_i} = \bN_{S_i}^{S_i} + \Xi_i $ where $\Theta_i, \Xi_i\geq 0 $. Let $B_{S_i} = G_i + B_{Y_i}|_{S_i} + \Theta_i \geq 0 $, then we have \[
        K_{\cF_{S_i}} + B_{S_i} + \bM_{S_i}^{S_i} := (K_{\cF_{Y_i}} + B_{Y_i}  + \bM_{Y_i}^i)|_{S_i},
        \]
    and \[
    \Delta_i = B_{S_i} + t_i(D_{S_i} + \Xi_i).
    \]
    \medskip

    \textit{Step 4.} In this step we assume that $\dim Q_i = 1$ for each $i$. Let $Q_i^\nu \to Q_i$ be the normalization of $Q_i$, and let $F_i$ be a general fiber of $S_i \to Q_i^\nu$. By the construction,
    \[
    K_{\cF_{S_i}} + \Delta_i + \mathbf{\Psi}^i_{S_i} = (K_{\cF_{Y_i}}+ B_{Y_i}+t_iD_{Y_i}+ \bM^i_{Y_i} + t_i \bN^i_{Y_i})|_{S_i} \equiv_{Q_i^\nu} 0.
    \]
    Since $F_i$ is general, we have that $(B_{Y_i}|_{S_i} )\cdot F_i \geq 0$ and $\Theta_i \cdot F_i\geq 0$. By Lemma~\ref{lemma:bnefonsurface}, $\bM_{S_i}^{S_i} \cdot F_i \geq 0$. Then $\bM^i_{Y_i}|_{S_i}\cdot F_i \geq 0$. Since $-T_i$ is ample over $X_i$, $T_i|_{S_i}\cdot F_i <0$, then 
    \[
    (D_{Y_i} + \bN_{Y_i}^i)|_{S_i}\cdot F_i  = - \frac{e_i}{t_i}(T_i|_{S_i})\cdot F_i >0.
    \] 
    It follows that $K_{\cF_{Y_i}}|_{S_i} \cdot F_i <0$. Since $F_i$ is general, $G_i\cdot F_i \geq 0$, which implies that $K_{\cF_{S_i}} \cdot F_i <0$. We may assume that $F_i$ is disjoint with $\Sing X \cup \Sing\cF$. Then by  
    \cite[Lemma 1.14]{Che22}, $F_i$ is $\cF_{S_i}$-invariant and $K_{\cF_{S_i}}\cdot F_i = -2$. It follows that 
    \[
        (\Delta_i+\mathbf{\Psi}_{S_i}^i)\cdot F_i = 2.
    \]

    Assume that $\bN_{S_i}^{S_i}\cdot F_i\neq 0$ for all $i$. Then we have $(\Delta_i + \bM^{S_i}_{S_i} + t_i \bN^{S_i}_{S_i})\cdot F_i = 2$. Since $F_i$ is general, we have $H_i \cdot F_i$, $ \mathbf{R}_{S_i}\cdot F_i \in \Z_{\geq 0}$ for every prime divisor $H_i$ contained in $\Supp \Delta_i$ and every b-Cartier b-nef b-divisor $\mathbf{R}$ appears in $\bM^{S_i}$ and $\bN^{S_i}$. Since $\Delta_i\in \frac{1}{2} K_+$ and $\bM^{S_i},\bN^{S_i}\in I$, by Lemma~\ref{lemma:finiteset}, there exists $\nu_i\in I\setminus \{0\}$, occurs as a coefficient of $\bN^{S_i}$, such that $t_i\nu_i$ belongs to a finite set that is independent of $i$, which contradicts with the fact that $\{t_i\}_{i=1}^\infty$ is strictly increasing.

    Now we may assume that $\bN^{S_i}_{S_i} \cdot F_i = 0$ for all $i$. 
    Let $\Delta_i(t) =  B_{S_i} + t(D_{S_i} + \Xi_i)$, then  $\Delta_i = \Delta_i(t_i)$. Note that
    \[
    (D_{S_i}+ \Xi_i)\cdot F_i = (D_{S_i} +  \bN_{S_i}^{S_i} + \Xi_i)\cdot F_i = (D_{Y_i} + \bN^i_{Y_i})|_{S_i} \cdot F_i >0,
    \]
    and 
    \begin{align*}
        (B_{S_i} + t_i(D_{S_i}+ \Xi_i) + \bM_{S_i}^{S_i})\cdot F_i = (\Delta_i(t_i) + \bM_{S_i}^{S_i})\cdot F_i = (\Delta_i + \mathbf{\Psi}_{S_i}^i)\cdot F_i = 2
    \end{align*}
    as $\bN^{S_i}_{S_i} \cdot F_i = 0$.
    Since $F_i$ is general, $H_i \cdot F_i$, $ \mathbf{R}_{S_i}\cdot F_i \in \Z_{\geq 0}$ for every prime divisor $H_i$ contained in $\Supp \Delta_i$ and every b-Cartier b-nef b-divisor $\mathbf{R}$ appears in $\bM^{S_i}$. Note that $\Delta_i \in \frac{1}{2}K_+$ and $\bM^{S_i}\in I$, hence by Lemma~\ref{lemma:finiteset}, there exists a linear function $f_i(t)$ of $t$, which occurs as a coefficient of $\Delta_i(t)$ and is not a constant function of $t$, such that $f_i(t_i)$ belongs to a finite set independent of $i$. Moreover, $f_i(t_i)\in \frac{1}{2}(K_i)_+ \setminus \{0\}$. We may write $2f_i(t_i) = \sum_{j=1}^n (a^i_j + t_ib^i_j)$, where $a^i_j,b^i_j\in I\cup \{0\}$ such that $\sum_j b^i_j \neq 0$. We may assume that $b_1^i \neq 0$ for each $i$. By Lemma~\ref{lemma:finiteset}, $a^i_1 + t_ib^i_1$ belongs to a finite set independent of $i$, which again contradicts with the fact that $\{t_i\}_{i=1}^\infty$ is strictly increasing.

    \medskip

    \textit{Step 5.} Now we may assume that $\dim Q_i = 0$ for all $i$. Then we have
    \[
        K_{\cF_{S_i}} + \Delta_i + \mathbf{\Psi}^i_{S_i}  \equiv 0.
    \]
    Let $L_i$ be a general leaf of $\cF_{S_i}$. Similarly to Step 4, we have
    \[
    \left(B_{Y_i}|_{S_i}+ \bM^i_{Y_i}|_{S_i} + G_i \right)\cdot L_i \geq 0.
    \]
    Since $-T_i$ is ample over $X_i$, $T_i|_{S_i}\cdot L_i<0$.
    Then $K_{\cF_{S_i}}\cdot L_i <0$ as $ (D_{Y_i} + \bN_{Y_i}^i)|_{S_i}\cdot L_i = -\frac{e_i}{t_i} T_i|_{S_i}\cdot L_i >0$. In particular, $K_{\cF_{S_i}}$ is not numerically trivial. By Lemma~\ref{lemma:surfacemodel}, there exists $\alpha_i: \Tilde{S_i} \to S_i$ such that $K_{\cG_i} + R_i = \alpha_i^* K_{\cF_{S_i}}$, where $\tilde{S_i}$ is $\Q$-factorial, $\cG_i = \alpha_i^{-1}\cF_{S_i}$ and $R_i\geq 0$ is exceptional over $S_i$. Note that $K_{\cF_{S_i}}$ may not be $\Q$-Cartier but we may define the pullback in the sense of Mumford. Let $\tilde{L_i}$ be the strict transform of $L_i$ on $\tilde{S_i}$. Then $R_i \cdot \tilde{L_i}\geq 0 $, and 
    \[
    K_{\cG_i} \cdot \tilde{L_i} = \alpha_i^*K_{\cF_{S_i}}\cdot \tilde{L_i} - R_i \cdot \tilde{L_i} = K_{\cF_{S_i}}\cdot L_i - R_i \cdot \tilde{L_i}<0.
    \]
    Since $K_{\cG_i} + R_i + \alpha_i^*(\Delta_i+ \mathbf{\Psi}^i_{S_i}) \equiv 0$, $R_i + \alpha_i^*(\Delta_i + \mathbf{\Psi}^i_{S_i}) $ is pseudo-effective, and $K_{\cG_i}$ is $\Q$-Cartier and not numerically trivial, we have $K_{\cG_i}$ is not pseudo-effective and then by \cite[Theorem 1.1]{CP19} or \cite[Lemma 2.21]{CS21}, $\cG_i$ is algebraically integrable and general leaves of $\cG_i$ are rational curves. Thus $\cF_{S_i}$ is also algebraically integrable, and general leaves of $\cF_{S_i}$ are rational curves.

    Let $L_{Y_i}$ be the image of $L_i$ on $Y_i$. We have
    \begin{equation}
        (K_{\cF_{Y_i}} + B_{Y_i} + t_iD_{Y_i} + \bM^i_{Y_i} + t_i \bN^i_{Y_i}) \cdot L_{Y_i} = 0. \tag{$\ast$}
    \end{equation}
    Assume that \[
    B_{Y_i} + t_iD_{Y_i} = \sum_
    {j=1}^{m_i} \alpha_{i,j}C_{i,j},\qquad \bM^i + t_i\bN^i = \sum_{l=1}^{n_i}\beta_{i,l} \mathbf{L}^{i,l},
    \]
    where $m_i,n_i \geq 0$, $C_{i,j}$ are prime divisors on $Y_i$ distinct from $T_i$, $\mathbf{L}^{i,l}$ are b-Cartier b-divisors on $Y_i$, $\alpha_{i,j}, \beta_{i,l} \in K_i \subset K$.
    By Lemma~\ref{lemma:forACC.adj.curves}, there exists a non-negative integer $\lambda_i, q_i$ and positive integers $p_{i,k}$ for $1 \leq k \leq q_i$ such that ($\ast$) can be written as 
    \[
    \frac{\lambda_i}{2} + \sum_{k=1}^{q_i} \frac{p_{i,k}-1}{p_{i,k}} + \sum_{j=1}^{m_i} \frac{\alpha_{i,j} }{2}(2C_{i,j}\cdot L_{Y_i}) + \sum_{l=1}^{n_i} \frac{\beta_{i,l}}{2} (2\bL^{i,l}_{Y_i} \cdot L_{Y_i}) = 2,
    \]
    and $(2C_{i,j}\cdot L_{Y_i})$ and $(2\bL^{i,l}_{Y_i} \cdot L_{Y_i})$ are non-negative integers. Let $\alpha_{i,j}' := \alpha_{i,j}$ if $(2C_{i,j}\cdot L_{Y_i})>0$ and $\alpha_{i,j}' := 0$ otherwise. Let $\beta_{i,l}' := \beta_{i,l}$ if $(2\bL^{i,l}_{Y_i} \cdot L_{Y_i}) > 0$ and $\beta_{i,l}' := 0$ otherwise.
    
    Note that \(
        \frac{1}{2}K \cup \left\{\,\frac{m-1}{m}\,|\, m \in \Z_{>0}\,\right\}
    \) is a DCC set. By Lemma~\ref{lemma:finiteset}, $\alpha_{i,j}', \beta_{i,l}'$ belong to a finite set $J$ independent of $i$. Note that $(D_{Y_i} + \bN^i_{Y_i}) \cdot L_{Y_i} = -\frac{e_i}{t_i} T_i \cdot L_{Y_i} >0$, there exists $a_i\in I\cup  \{0\}$, $b_i \in I \setminus \{0\}$, such that $a_i+ t_ib_i$, which equals to one of the $\alpha_{i,j}'$ and $\beta_{i,l}'$, belongs to the finite set $J$. It follows that $\{t_i\}$ belongs to an ACC set independent of $i$ which leads to a contradiction.
\end{proof}

\bibliographystyle{alpha}
\bibliography{ref}

\end{document}